\documentclass[11pt]{article}

\title{\textbf{Haystack Hunting Hints and Locker Room Communication}\thanks{Research partially supported by the Centre for Discrete Mathematics and its Applications (DIMAP), by an IBM Faculty Award, by EPSRC award EP/V01305X/1, and by an EPSRC Doctoral Training Partnership.}}

\author{Artur Czumaj
	\thanks{Department of Computer Science and Centre for Discrete Mathematics and its Applications (DIMAP), University of Warwick. Email: A.Czumaj@warwick.ac.uk.
}
    \and
    George Kontogeorgiou
	\thanks{Mathematics Institute, University of Warwick. Email: George.Kontogeorgiou@warwick.ac.uk.}
    \and
    Mike Paterson
	\thanks{Department of Computer Science and Centre for Discrete Mathematics and its Applications (DIMAP), University of Warwick. Email: M.S.Paterson@warwick.ac.uk.
}
}

\date{{\small \today, \zeit{}}}
\date{}

\usepackage{amsmath,amsfonts,amsthm,amssymb}
\allowdisplaybreaks
\usepackage{latexsym}
\usepackage{mdframed}
\usepackage{fancybox}
\usepackage{fullpage}
\usepackage{paralist}
\usepackage{environ}
\usepackage{xspace}
\usepackage{graphicx}

\usepackage{ifpdf}
\newcommand{\mydriver}{hypertex}
\ifpdf
 \renewcommand{\mydriver}{pdftex}
\fi
\usepackage[breaklinks,\mydriver]{hyperref}

\usepackage{cleveref}

\newcount\shortyear\newcount\shorthour\newcount\shortminute
\shorthour=\time\divide\shorthour by 60\shortyear=\shorthour
\multiply\shortyear by 60\shortminute=\time\advance\shortminute by
-\shortyear \shortyear=\year\advance\shortyear by -1900

\def\zeit{\number\shorthour:\ifnum\shortminute<10 0\number\shortminute
	\else\number\shortminute\fi}

\newcommand{\COMMENTED}[1]{{}}
\newcommand{\junk}[1]{\COMMENTED{#1}}

	\newtheorem{theorem}{Theorem}
	\newtheorem{lemma}[theorem]{Lemma}
	\newtheorem{claim}[theorem]{Claim}
	\newtheorem{remark}[theorem]{Remark}
	\newtheorem{example}{Example}
	\newtheorem{conjecture}{Conjecture}

	\newtheorem{defpart}[theorem]{Definition}
	\newenvironment{definition}{\begin{defpart}\sl}{\end{defpart}}

\renewcommand{\Pr}[1]{\ensuremath{\mathbf{Pr}[#1]}}
\newcommand{\PPr}[1]{\ensuremath{\mathbf{Pr}\big[#1\big]}}
\newcommand{\Ex}[1]{\ensuremath{\mathbf{E}[#1]}}
\newcommand{\EEx}[1]{\ensuremath{\mathbf{E}\big[#1\big]}}

\newcommand{\Var}[1]{\ensuremath{\mathbf{Var}[#1]}}
\newcommand{\Cov}[1]{\ensuremath{\mathbf{Cov}[#1]}}

\def\epsilon{\ensuremath{\varepsilon}}
\newcommand{\PER}{\ensuremath{\mathbb{S}_n}\xspace}

\newcommand{\set}{\ensuremath{\Phi}}
\newcommand{\eset}{\ensuremath{\mathcal{P}}}
\newcommand{\esetc}{\ensuremath{\mathcal{PC}}}
\renewcommand{\esetc}{\ensuremath{\textsf{PC}}}
\newcommand{\fixed}{\ensuremath{\mathcal{F}}}
\newcommand{\nset}{\ensuremath{[n\mathord{-}1]}}

\newcommand{\Strat}{\ensuremath{\mathbb{C}}}

\newcommand{\magn}{\ensuremath{\textsf{\small mag}}}
\newcommand{\mmag}{\ensuremath{\mu}}
\renewcommand{\mmag}{\ensuremath{\textsf{\small max-\magn}}}
\newcommand{\mint}{\ensuremath{\textsf{\small int}}}

\newcommand{\hint}{\ensuremath{\textcolor[rgb]{0.50,0.00,1.00}{\mathfrak{h}}}\xspace} 
\newcommand{\choos}{\ensuremath{\textcolor[rgb]{0.50,0.00,1.00}{\mathfrak{i}}}\xspace} 
\newcommand{\sought}{\ensuremath{\textcolor[rgb]{0.50,0.00,1.00}{\mathfrak{s}}}\xspace} 

\renewcommand{\hint}{\ensuremath{\mathfrak{h}}\xspace} 
\renewcommand{\choos}{\ensuremath{\mathfrak{i}}\xspace} 
\renewcommand{\sought}{\ensuremath{\mathfrak{s}}\xspace} 

\newcommand{\Artur}[1]
{{{\footnote{{\sc\small \textcolor[rgb]{0.50,0.00,1.00}{\textbf{Artur:}}} \textcolor[rgb]{0.00,0.07,1.00}{#1}}}}}
\newcommand{\George}[1]
{{{\footnote{{\sc\small \textcolor[rgb]{1.00,0.50,0.00}{\textbf{George:}}} #1}}}}
\newcommand{\Mike}[1]
{{{\footnote{{\sc\small \textcolor[rgb]{0.00,1.00,0.50}{\textbf{Mike:}}} #1}}}}
\newcommand{\LR}{\emph{locker room}\xspace}
\newcommand{\LRP}{\emph{locker room} problem\xspace}

\newlength{\savedparindent}
\newcommand{\SaveIndent}{\setlength{\savedparindent}{\parindent}}
\newcommand{\RestoreIndent}{\setlength{\parindent}{\savedparindent}}

\newcommand{\InGray}[1]{%
	\SaveIndent{} %
	\centerline{ \fcolorbox[rgb]{0,0,0}{0.95,0.95,0.95}{
			\begin{minipage}{0.70\linewidth} %
				\RestoreIndent{}
				\small#1
			\end{minipage}
} } }

\begin{document}

\maketitle

\begin{abstract}
We want to efficiently find a specific object in a large unstructured set, which we model by a \emph{random $n$-permutation}, and we have to do it by revealing just a single element. Clearly, without any help this task is hopeless and the best one can do is to select the element at random, and achieve the success probability $\frac1n$.
Can we do better with some small amount of advice about the permutation, even without knowing the object sought? We show that by providing advice of just one integer in $\{0,1,\dots,n\mathord{-}1\}$, one can improve the success probability considerably, by a $\Theta(\frac{\log n}{\log\log n})$ factor.

We study this and related problems, and show asymptotically matching upper and lower bounds for their optimal probability of success. Our analysis relies on a close relationship of such problems to some intrinsic properties of random permutations related to the rencontres number.
\end{abstract}


\section{Introduction}
\label{sec:intro}

Understanding basic properties of random permutations is an important concern in modern data science. For example, a preliminary step in the analysis of a very large data set presented in an unstructured way is often to model it assuming the data is presented in a random order.
Understanding properties of random permutations would guide the processing of this data and its analysis.
In this paper, we consider a very natural problem in this setting. You are given a set of $n$ objects (\nset, say\footnote{Throughout the paper we use the standard notation $\nset := \{0,\dots,n\mathord{-}1\}$, and we write $\log$ for $\log_2$.}) stored in locations $x_0,\ldots,x_{n\mathord{-}1}$ according to a random permutation $\sigma$ of $\nset$.
This is the \emph{haystack}, and you want to find one specific object, not surprisingly called the \emph{needle}, by drawing from just one location.

Clearly, the probability of finding this object \sought in a single draw is always $\frac1n$ (whichever location you choose, since the permutation $\sigma$ is random, the probability that your object is there is exactly~$\frac1n$).
But can I give you any advice or \emph{hint} about $\sigma$ --- \emph{without knowing which object you are seeking} --- to improve the chance of you finding \sought? If I could tell you the entire $\sigma$ (which can be encoded with $\log(n!) = \Theta(n\log n)$ bits) then this task is trivial and you would know the location of \sought.
But what if I give you just a small hint (on the basis of $\sigma$), one number $\hint$ from $\nset$ (or equivalently, one $\log n$-bit sequence) --- even when I know nothing about the object sought?

Formally, the goal is to design a strategy to choose a \emph{hint} $\hint = \hint(\sigma)$ and an \emph{index} $\choos = \choos(\hint,\sought)$, with both $\hint, \choos \in \nset$, such that for a given $\sought \in \nset$, $\Pr{\sigma(\choos) = \sought}$ is maximized, where the probability is over the random choice of $\sigma$ and over the randomness in the choice of the strategy (since $\hint = \hint(\sigma)$ and $\choos = \choos(\hint,\sought)$ may be randomized functions), see also \Cref{sec:setup-random-sought}.


\subsection{Related puzzle: \emph{communication in the locker room}}

The \emph{needle in a haystack} problem is closely related to the following \LRP (see \Cref{fig:cards}):
The locker room has $n$ lockers, numbered $0, \dots, n\mathord{-}1$.
A set of $n$ cards, numbered $0, \dots, n\mathord{-}1$, is inserted in the lockers according to a uniformly random permutation $\sigma$.
Alice and Bob are a team with a task. Alice enters the locker room, opens all the lockers and can swap the cards between just two lockers, or may choose to leave them unchanged. She closes all the lockers and leaves the room.
Bob is given a
number $\sought \in \nset$ and his task is to find card $\sought$. He can open at most two lockers.
Before the game begins, Alice and Bob may communicate to decide on a strategy.
What is their optimal strategy, and how efficient is it?

As in the \emph{needle in a haystack} problem, without help from Alice, Bob can do no better than open lockers at random.
If he opens one locker his probability of success is $\frac1n$ and if he opens two lockers this probability is $\frac2n$.
With the help of Alice, he can do better when opening one locker. E.g., their strategy could be that Bob will open locker \sought,
where \sought is his given number. Alice would then try to increase the number of fixed points in the permutation above the expected number of 1.
If there is a transposition she can reverse it, increasing the number of fixed points by two,
and if not she can produce one more fixed point (unless the permutation is the identity).
This strategy succeeds with probability just under $\frac{12}{5n}$.
When Bob can open two lockers, the challenge is to increase the success probability to~$\omega(\frac1n)$.

The answer involves viewing Bob's first locker opening in a different way: not as looking for his card but as receiving a communication from Alice.
The interest is in finding what kind of information Alice can send about the permutation which could help Bob in his search.


 \medskip

Now, we invite the reader to stop for a moment: to think about this puzzle, to find any strategy that could ensure the success probability would be $\omega(\frac1n)$.
\begin{figure}[h]
\centerline{\includegraphics[scale=0.99]{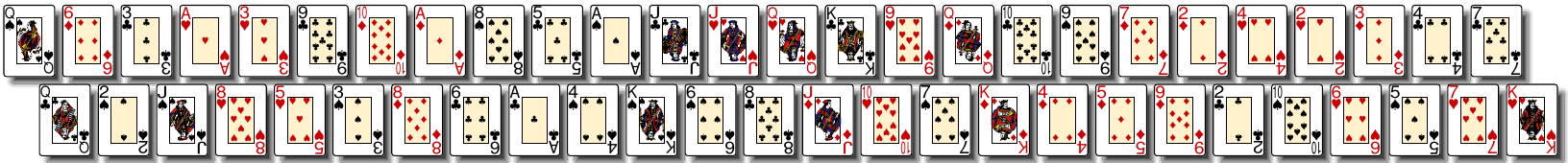}}
\caption{
Consider the following randomly shuffled deck, one card per locker. What advice should Alice give to Bob --- just by swapping the locations of at most one pair of cards --- to increase the probability that Bob will find his randomly chosen card by opening at most two lockers?}\label{fig:cards}
\end{figure}

It is easy to see that a solution to the \emph{needle in a haystack} search problem immediately yields a solution to the \LRP:
Alice just takes the card corresponding to the advice and swaps it into the first locker.
For example, the shuffled deck from \Cref{fig:cards} corresponds to the following permutation $\sigma$ of 52 numbers:
\begin{align*}
	\sigma(0, 1, \dots, 51) =
	\langle &
	49, 17, 1, 38, 27, 7, 21, 25, 45, 3, 51, 9, 35, 36, 11, 33, 23, 8, 46, 18, 13, 28, 26, 14, 2, 5, \\
    & 10, 39, 48, 32, 29, 40, 19, 4, 50, 43, 6, 22, 34, 44, 24, 15, 16, 20, 0, 47, 30, 42, 31, 37
	\rangle
\end{align*}
with mapping:
$\clubsuit$: 0--12 (in order \emph{2,3,4,5,6,7,8,9,10,J,Q,K,A}),
$\diamondsuit$: 13--25, 
$\heartsuit$: 26--38, 
$\spadesuit$: 39--51. 
We see, for example, that $\spadesuit$Q, card number~$49$ is in locker~$0$.
If in the \emph{needle in a haystack} search problem the advice is a number $\hint \in \nset$, then Alice swaps the contents of locker~$0$ and the locker containing the card corresponding to number $\hint$.
This way, Bob gets the advice $\hint$ by opening locker~$0$.

For the strategy we propose in Theorem~\ref{thm:upper}, Alice would swap $\spadesuit$Q and $\heartsuit$5.
But can we do better?


\subsection{Results for the \emph{needle in a haystack} and \LR problems}
\label{subsec:results}

We present a tight analysis of the \emph{needle in a haystack} search problem. While some basic examples suggest that it is difficult to ensure success probability $\omega(\frac1n)$, we will show that one can improve this probability considerably.
Our main results are tight (up to lower order terms) lower and upper bounds for the maximum probability that with a single number hint one can find the object sought. First, we will show that for any strategy, the probability that one can find the sought object is at most $\frac{(1+o(1)) \log n}{n \log\log n}$ (\Cref{thm:upper}).
Next, as the main result of this paper, we will complement this by designing a simple strategy that with a hint ensures that the sought object is found with probability at least $\frac{(1+o(1)) \log n}{n \log\log n}$ (\Cref{thm:lower}).

Further, we demonstrate essentially the same results for the \LRP. \Cref{thm:lower} for the \emph{needle in a haystack} search problem immediately implies that there is a simple strategy for Alice and Bob which ensures that Bob finds his card with probability at least $\frac{(1+o(1)) \log n}{n \log\log n}$. We will complement this claim, and extend in \Cref{Alice-Bob-vs-needle-haystack} the result from \Cref{thm:upper} for the \emph{needle in a haystack} search problem, to prove that for any strategy for Alice and Bob, the probability that Bob finds the required card is at most $O\left(\frac{\log n}{n \log\log n}\right)$.

\paragraph{Techniques.}
Our analysis exploits properties of random permutations to ensure that some short advice can reveal information about the input permutation, which can be used to increase the success probability substantially.
Our approach relies on a close relationship between the \emph{needle in a haystack} search problem and some intrinsic properties of random permutations related to the \emph{rencontres number}, the number of $n$-permutations with a given number of fixed points.
%

To show the upper bound for the success probability (\Cref{thm:upper}), we observe that every deterministic strategy corresponds to a unique partition of $\PER$ (set of all permutations of \nset) into $n$ parts, with part \hint containing exactly those permutations that cause the choice of hint \hint. By a careful analysis of the properties of this partition, we devise a metric for the best possible accuracy of the prediction, counting instances in each part of the partition in which a permutation maps a given choice \choos to \sought. By combining these estimates with the bounds for the rencontres number, we prove the upper bound for the success probability in the \emph{needle in a haystack} search problem. An application of Yao's principle shows that our results are also valid for randomized strategies.

To show the lower bound for the success probability (\Cref{thm:lower}), we present  a simple \emph{shift strategy}, and then provide a non-trivial analysis of random permutations that demonstrates desirable properties of this strategy. The analysis here is related to the maximum load problem for balls and bins, where one allocates $n$ balls into $n$ bins, chosen independently and uniformly at random (\emph{i.u.r.}).
However, the dependencies between locations of distinct elements in the random permutations make this analysis more complex (see \Cref{remark:balls-and-bins} for more detailed discussion).

Finally, while a solution to the \emph{needle in a haystack} search problem immediately yields a solution to the \LRP with the same success probability, we complement our analysis by showing (\Cref{Alice-Bob-vs-needle-haystack}) that no strategy of Alice and Bob can do much better. We show that Alice can do little more than just to send a few numbers to Bob, which is essentially the setup of the \emph{needle in a haystack} search problem.


\subsection{Background: Permutations, puzzles, and locker rooms}

Our \LRP follows a long line of the study of combinatorial puzzles involving the analysis of properties of permutations. One such example is the following locker problem involving prisoners and lockers:
There are $n$ lockers into which a random permutation of $n$ cards are inserted. Then $n$ prisoners enter the locker room one at a time and are allowed to open half of the lockers in an attempt to find their own card.
The team of prisoners wins if every one of them is successful. The surprising result is that there is a strategy which wins with probability about $1 - \ln 2$.
This problem was initially considered by Peter Bro Miltersen and appeared in his paper with Anna G\'{a}l \cite{GM03}, which won a best paper award at ICALP 2003.
In that paper they refer to a powerful strategy approach suggested by Sven Skyum but it was left to the readers to find it for themselves.
This is the idea of using the number contained in each locker as a pointer to another locker. Thus using a sequence of such steps corresponds to following a cycle in the permutation.
Solutions to these problems are of a \emph{combinatorial and probabilistic flavor} and involve an \emph{analysis of the cycle structure of random permutations}.
The original paper \cite{GM03} stimulated many subsequent papers considering different variants (see, e.g., \cite{But02,GS05}), including a matching upper bound provided in \cite{CW06}.
An early version giving the problem where each prisoner can open half of the lockers was published by \cite{Win06} (see also \cite[p.~18]{Win07}).
If each prisoner begins with the locker corresponding to the number they seek then they will all succeed provided that there is no cycle in the permutation which is longer than $\frac{n}{2}$. It is easy to show that a helpful prison warder, Alice, can always find an appropriate transposition of the contents of two lockers so that the resulting permutation has no cycle longer than $\frac{n}{2}$.
We were told of this observation recently by Kazuo Iwama and this stimulated the current paper, in which we subvert the locker problem tradition with a problem which has little to do with the cycle structure of
permutations and is more concerned with some basic communication complexity and rather different properties of permutations.

Various results about permutations have found diverse applications in computer science, especially for sorting algorithms (for example, see \cite[Chapter~5]{Knu98}). In this paper, we are particularly interested in two such questions. Firstly, to apply known results concerning the asymptotic growth of the rencontres numbers, in order to approximate the optimal success probabilities in both the \emph{needle in a haystack} problem and the \LRP. Secondly, to use the concept of the rencontres numbers to examine the way in which the sizes of ``shift sets'' (sets of elements which a permutation displaces by the same number of positions ``to the right'') are distributed in permutations of $\PER$ for a fixed natural number $n$. In particular, to determine the mean size of the largest shift set of a permutation chosen uniformly at random from $\PER$, as well as to show that it is typical, i.e., that the variance of the size of the largest shift set is small. These results are useful for providing a concrete optimal strategy for both of the titular search problems.


\section{Preliminaries}
\label{sec:prelim}


\subsection{Formal framework and justification about worst-case vs. random \sought}
\label{sec:setup-random-sought}

\newcommand{\pp}{\ensuremath{\textcolor[rgb]{0.50,0.00,1.00}{\mathfrak{p}}}\xspace}
\renewcommand{\pp}{\ensuremath{\mathfrak{p}}\xspace}

We consider the problem with two inputs: a number $\sought \in \nset$ and a permutation $\sigma \in \PER$. We are assuming that $\sigma$ is a random permutation in \PER; no assumption is made about \sought.

For the \emph{needle in a haystack} search problem (a similar framework can be easily set up for the \LRP), a \emph{strategy} (or an \emph{algorithm}) is defined by a pair of (possibly randomized) functions, $\hint = \hint(\sigma)$ and $\choos = \choos(\hint,\sought)$, with both $\hint, \choos \in \nset$.

For a fixed strategy, let $\pp(\sought)$ be the success probability for a given \sought and for a randomly chosen $\sigma \in \PER$. That is,
\begin{align*}
	\pp(\sought) &= \Pr{\sigma(\choos) = \sought}
	\enspace,
\end{align*}
where the probability is over $\sigma$ taken i.u.r. from \PER, and over the randomness in the choice of the strategy (since both $\hint = \hint(\sigma)$ and $\choos = \choos(\hint,\sought)$ may be randomized functions).

The goal is to design an algorithm (find a strategy) that will achieve some given success probability for every $\sought \in \nset$.
That is, we want to have a strategy which maximizes
\begin{align*}
	\Pr{\mathcal{V}} &= \min_{\sought \in \nset}\{\pp(\sought)\}
    \enspace.
\end{align*}

In our analysis for the upper bounds in \Cref{sec:prelim,sec:upper} (\Cref{thm:upper}) and \Cref{sec:locker-room-analysis} (\Cref{Alice-Bob-vs-needle-haystack}), for simplicity, we  will be making the assumption that \sought (the input to the \emph{needle in a haystack} search problem and to the \LRP) is random, that is, \sought is chosen i.u.r. from \nset. (We do not make such assumption in the lower bound in \Cref{sec:lower} (\Cref{thm:lower}), where the analysis is done explicitly for arbitrary \sought.) Then the main claim (\Cref{thm:upper}) is that if we choose \sought i.u.r. then $\pp(\sought) \le \frac{(1+o(1)) \log n}{n\log\log n}$. Observe that one can read this claim equivalently as that $\sum_{\sought \in \nset} \frac{\pp(\sought)}{n} \le \frac{(1+o(1)) \log n}{n\log\log n}$. However, notice that this trivially yields
\begin{align*}
	\Pr{\mathcal{V}} &= \min_{\sought \in \nset}\{\pp(\sought)\} \le
	\sum_{\sought \in \nset} \frac{\pp(\sought)}{n}
	\enspace,
\end{align*}
and therefore \Cref{thm:upper} yields $\Pr{\mathcal{V}} \le \frac{(1+o(1)) \log n}{n\log\log n}$, as required.

Note that such arguments hold only for the upper bound. Indeed, since $\min_{\sought \in \nset}\{\pp(\sought)\}$ may be much smaller than $\sum_{\sought \in \nset} \frac{\pp(\sought)}{n}$, in order to give a lower bound for the success probability, \Cref{thm:lower} proves that there is a strategy that ensures that $\pp(\sought) \ge \frac{(1+o(1)) \log n}{n\log\log n}$ for every $\sought \in \nset$; this clearly yields $\Pr{\mathcal{V}} \ge \frac{(1+o(1)) \log n}{n\log\log n}$, as required.


\subsection{Describing possible strategies for \emph{needle in a haystack}}
\label{subsec:strategy}

In this section, we prepare a framework for the study of strategies to prove an upper bound for the success probability for the \emph{needle in a haystack} search problem (see \Cref{sec:upper}). For simplicity, we will consider (in \Cref{sec:prelim,sec:upper,sec:locker-room-analysis}) the setting when \sought is chosen i.u.r. from \nset; see \Cref{sec:setup-random-sought} for justification that this can be done without loss of generality. First, let us rephrase the original problem in a form of an equivalent communication game between Alice and Bob: Bob, the \emph{seeker}, has as his input a (random) number $\sought \in \nset$. Alice, the \emph{adviser}, sees a permutation $\sigma$ chosen i.u.r. from \PER, and uses $\sigma$ to send advice to Bob in the form of a number $\hint \in \nset$. Bob does not know $\sigma$, but on the basis of \sought and \hint, he picks some $\choos \in \nset$ trying to maximize the probability that $\sigma(\choos) = \sought$.

First we will consider \emph{deterministic strategies} (we will later argue separately that randomized strategies cannot help much here). Since we consider deterministic strategies, the advice sent is a function $\PER \rightarrow \nset$, which can be defined by a partition of \PER into $n$ sets. This naturally leads to the following definition of a \emph{strategy}.

\begin{definition}
\label{def:strategy}
A \textbf{strategy} $\Strat$ for $\PER$ is a partition of $\PER$ into $n$ sets $C_0, C_1, \dots, C_{n-1}$. Such a strategy $\Strat$ is denoted by $\Strat = \langle C_0, C_1, \dots, C_{n-1}\rangle$.
\end{definition}

Given a specific strategy $\Strat$, we examine the success probability. Let $\mathcal{V}$ be the event that the sought number is found, $\mathcal{A}_{\hint}$ the event that $\hint$ is the received advice, and $\mathcal{B}_{\sought}$ the event that $\sought$ is the sought number.
Notice that for every $\hint \in \nset$ we have $\Pr{\mathcal{A}_{\hint}} = \frac{|C_{\hint}|}{n!}$ and for every $\sought \in \nset$ we have
$\Pr{\mathcal{B}_{\sought}} = \frac1n$. Therefore, since the events $\mathcal{A}_{\hint}$ and $\mathcal{B}_{\sought}$ are independent,
\begin{align}
    \Pr{\mathcal{V}} &=
    \sum_{\sought=0}^{n-1}\sum_{\hint=0}^{n-1}
        \Pr{\mathcal{V} | \mathcal{A}_{\hint} \cap \mathcal{B}_{\sought}} \cdot
        \Pr{\mathcal{A}_{\hint} \cap \mathcal{B}_{\sought}}
            =
    \sum_{\sought=0}^{n-1}\sum_{\hint=0}^{n-1}
        \Pr{\mathcal{V} | \mathcal{A}_{\hint} \cap \mathcal{B}_{\sought}} \cdot
        \Pr{\mathcal{A}_{\hint}} \cdot \Pr{\mathcal{B}_{\sought}}
            \nonumber\\&=
    \frac1n \sum_{\hint=0}^{n-1} \frac{|C_{\hint}|}{n!} \cdot \sum_{\sought=0}^{n-1}
        \Pr{\mathcal{V} | \mathcal{A}_{\hint} \cap \mathcal{B}_{\sought}}
    \enspace.
\label{bound:prob}
\end{align}

\begin{definition}
\label{def:magneticity}
Let $\Strat = \langle C_0, C_1, \dots, C_{n-1}\rangle$ be a strategy.
The \textbf{magneticity} of an element~$i$ for an element~$k$ in the class $C_j$ is defined
as $\magn(C_j,i,k) = |\{\sigma \in C_j: \sigma(i) = k \}|$.

The element with the greatest magneticity for~$k$ in the class $C_j$ is called the \textbf{magnet in $C_j$ of~$k$} and is denoted $\mmag(C_j,k)$; ties are broken arbitrarily.
The magneticity of $\mmag(C_j,k)$ is called the \textbf{intensity of~$k$ in $C_j$}, denoted by $\mint(C_j, k)$; that is, $\mint(C_j,k) = \max_{i \in \nset} \{\magn(C_j,i,k)\}$.
\end{definition}
This can be extended in a natural way to any $C = \langle A_0, A_1, \dots, A_{n-1}\rangle$ of $n$ subsets of~\PER.

Let us discuss the intuitions. Firstly, the \emph{magneticity} in the class $C_j$ of an element~$i$ for an element~$k$, $\magn(C_j,i,k)$, denotes the number of permutations in $C_j$ with $k$ in position $i$. Therefore, the \emph{magnet} in $C_j$ of~$k$ is an index $i \in \nset$ such that, among all permutations in $C_j$, $k$ is most likely to be in position~$i$. The \emph{intensity} in $C_j$ of~$k$ denotes just the number of times (among all permutations in $C_j$) that~$k$ appears in the position of the magnet~$i$.

In the \emph{needle in a haystack} search problem, Alice sends to Bob a message $\hint$ which points to a class $C_{\hint}$ of their agreed strategy $\Strat$, and Bob has to choose a number $\choos$ in order to find whether $\sigma(\choos)$ is the number $\sought \in \nset$ which he seeks. The maximum probability that they succeed is $\frac{\mint(C_{\hint}, \sought)}{|C_{\hint}|}$, realized if Bob opts for the magnet of $\sought$ in $C_{\hint}$. Thus, by (\ref{bound:prob}), we obtain
\begin{align*}
    \Pr{\mathcal{V}} &\le
            \frac{1}{n} \cdot \frac{1}{n!} \sum_{\substack{\sought, \hint \in \nset}} \mint(C_{\hint}, \sought)
    \enspace.
\end{align*}

\begin{definition}
Let the \textbf{field} of \PER be $F(n) = \max_{\Strat = \langle C_0, C_1, \dots, C_{n-1}\rangle} \sum_{\substack{\sought, \hint \in \nset}} \mint(C_{\hint}, \sought)$.
\end{definition}
With this definition, a strategy which yields the field of \PER is called \emph{optimal}, and
\begin{align}
\label{upper-bound:prob}
    \Pr{\mathcal{V}} &\le
    \frac{1}{n} \cdot \frac{1}{n!} \sum_{\substack{\sought, \hint \in \nset}} \mint(C_{\hint}, \sought)
        \le
    \frac{1}{n} \cdot \frac{F(n)}{n!}
    \enspace.
\end{align}

We will use this bound to prove \Cref{thm:upper} in \Cref{sec:upper}, that whatever the strategy, we always have $\Pr{\mathcal{V}} \le \frac{(1+o(1)) \cdot \log n}{n  \log\log n}$.


\junk{We will denote by $\PER$ the group of permutations over the set $\nset := \{0, \dots, n-1\}$.
	
	\Artur{This and the following sections will have to be modified. For example, for the results in \ref{thm:upper,thm:lower} we don't have to refer to Alice and Bob at all.}
	
	For a successful strategy Alice must be able to transmit information to Bob. Since the game ends immediately after Bob opens his second chosen locker, the only information which is relevant to Bob's decision-making is that given by the content of his first chosen locker. Therefore the first locker which Bob opens should probably contain a message from Alice.
	
	Through her allowed transposition, Alice can change the content of only two lockers.
	To ensure that Bob opens a locker which contains a message from Alice, and that he recognizes the message as such, it seems that Alice and Bob should agree in advance the locker for Bob to open first, without loss of generality, locker~0 say.
	Alice's message will necessarily be very simple: the number of a card. It remains to see what sort of information could be useful.
	
	
\subsection{Two-party communication setting}
\label{subsec:com-compl}
	
With the above observation at hand, for some readers it might be more natural to consider the problem in a basic two-party communication setting. If the only advice Alice can provide is a transposition of the cards in two lockers, involving a specific locker (locker 0), the problem is closely related to the following two-party communication protocol:
	
\begin{itemize}
\item Alice receives as input a random permutation $\sigma \in \PER$;
\item Bob gets as input a random number $i \in \nset$ and is searching for $k$ with $\sigma(k) = i$;
\item Alice sends Bob a number in $j \in \nset$;
\item Bob selects $k$ on the basis of $j$;
\item Alice and Bob's goal is to maximize $\Pr{\sigma(k)=i}$.
\end{itemize}
	
The difference here is rather small, and in fact it is not difficult to see that the success probability differs from the one for the original locker problem by at most an additive term of $O(\frac1n)$.}


\junk{\subsection{Examples of strategies}
	\label{subsec:strategies}

	Having defined the notion of a strategy (for Alice and Bob) and its basic properties, let us present some concrete examples of strategies:
	\begin{enumerate}[(i)]
		\item The \emph{freshman's strategy} is to group permutations according to the image of $0$. That is, $\sigma, \sigma'$ belong to the same class if and only if $\sigma(0) = \sigma'(0)$. This is a natural strategy to conceive, and it agrees with the common (erroneous) notion that efficiency cannot be improved beyond $O(\frac{1}{n})$. Indeed, straightforward calculations with the above tools yield a success probability of $\frac{3}{n}-o(\frac{1}{n})$ for the freshman's strategy. In what follows, we will show this to be far from optimal.
		\item The \emph{shift strategy} is to group permutations according to their most common shift. That is, $\Strat = \langle C_0, C_1, \dots, C_{n-1} \rangle$ such that if $\sigma \in C_i$ then for every $j \in \nset$ it holds
		\begin{align*}
		|\{\ell \in \nset: \ell - \sigma(\ell) = i \pmod n\}| \ge |\{\ell \in \nset: \ell - \sigma(\ell) = j \pmod n\}|
		\enspace.
		\end{align*}
		This strategy is studied in detail in \Cref{sec:lower}, where it is shown to be an asymptotically optimal strategy (see \Cref{thm:lower}) and satisfies $\Pr{\mathcal{V}} \ge \frac{(1+o(1)) \cdot \log n}{n  \log\log n}$.
\end{enumerate}}


\subsection{Derangements}
\label{subsec:derangements}

We use properties of random permutations related to derangements and rencontres numbers.

\begin{definition}
\label{def:derangements}
A permutation $\sigma \in \PER$ with no fixed points is called a \textbf{derangement}. The number of derangements in $\PER$ is denoted $D_n$. A permutation $\sigma\in \PER$ with exactly $r$ fixed points is called an \textbf{$r$-partial derangement}. The number of $r$-partial derangements in $\PER$ (also known as the \emph{rencontres number}) is denoted $D_{n,r}$.
\end{definition}

\Cref{def:derangements} yields $D_{n,0} = D_n$ and it is easy to see that $D_{n,r} = \binom{n}{r} \cdot D_{n-r}$. It is also known (see, e.g., \cite[p.~195]{GKP94}) that $D_n = \lfloor \frac{n!}{e} + \frac12 \rfloor$, and hence one can easily show
$D_{n,r} \le \frac{n!}{r!}$.
\junk{
	\begin{align}
	\label{bound-for-Dnk}
	D_{n,r} &\le
	\frac{n!}{r!}
	\enspace.
	\end{align}
}
%

\junk{We shall also make use of \emph{generalised derangements} and \emph{generalised rencontres numbers}:
	
	\begin{definition}
		\label{def:generalised_derangements}
		Given natural numbers $m$, $n$, let $\PER^{(m)}$ be the permutations of the symbols $0,...,n-1,u_0,...,u_{m-1}$. Obviously, $|\PER^{(m)}|=(m+n)!$. A fixed point of a permutation $\sigma=a_0...a_{n+m-1}\in \PER^{(m)}$ is an element $i$ with $\sigma(i)=a_i$. A \textbf{generalized derangement} is a permutation of $\PER^{(m)}$ without fixed points. The number of generalized derangements is denoted $D_n^{(m)}$. A permutation $\sigma\in \PER^{(m)}$ with exactly $r$ fixed points is called a \textbf{generalised $r$-partial derangement}. The number of generalised $r$-partial derangements (also known as the generalized rencontres number) is denoted $D_{n,r}^{(m)}$.
	\end{definition}
	
	The generalized rencontres numbers have been studied extensively in \cite{CFMS18}. The following inequality holds:
	
	\begin{align}
	\label{bound-for-Dnk^{(m)}}
	D_{n,r}^{(m)} &\le
	\frac{(m+n)!}{r!}
	\enspace.
	\end{align}}


\section{Upper bound for the success probability for needle in a haystack}
\label{sec:upper}

We will use the framework set up in the previous section, in particular the tools in \Cref{def:magneticity} and inequality~(\ref{upper-bound:prob}) and that \sought is chosen i.u.r. from \nset, to bound from above the best possible success probability in the \emph{needle in a haystack} search problem.

\begin{theorem}
\label{thm:upper}
For any strategy in the \emph{needle in a haystack} problem, the success probability~is
\begin{align*}
    \Pr{\mathcal{V}} &\le
    \frac{(1+o(1)) \cdot \log n}{n  \log\log n}
    \enspace.
\end{align*}
\end{theorem}
	
\begin{proof}
We will first consider only \emph{deterministic} strategies
and, only at the end, we will argue that this extends to randomized strategies.


Consider an optimal strategy $\Strat = \left< C_0, \dots, C_{n-1} \right>$. First, we will modify sets $C_0, \dots, C_{n-1}$ to ensure that each $C_j$ has $n$ distinct magnets.

Fix $j \in \nset$. Suppose that there are two elements $k_1 < k_2 \in \nset$ with the same magnet $i_1$ in $C_j$. Since there are exactly $n$ elements and $n$ possible magnets, there is some~$i_2\in \nset$ which is not a magnet in $C_j$ of any element. For every $\sigma\in C_j$ with $\sigma(i_1)=k_2$, calculate $\sigma' =\sigma(i_1i_2)$ (that is, $\sigma'$ is the same as $\sigma$, except that the images of $i_1$ and $i_2$ are exchanged). Now, if $\sigma' \notin C_j$, then remove $\sigma$ from $C_j$ and add $\sigma'$ to $C_j$.
We notice the following properties of the resulting set $C_j'$ in the case that some $\sigma$ is replaced by $\sigma'$:
\begin{enumerate}[(i)]
\item $|C_j'|=|C_j|$.
\item $i_2$ can be chosen as the new magnet of $k_2$. Indeed, for every $i \ne i_1, i_2$, we have
    \begin{align*}
        \magn(C_j',i_2,k_2) &> \magn(C_j,i_1,k_2) = \mint(C_j,k_2) \ge \magn(C_j,i,k_2) = \magn(C_j',i,k_2) , \mbox{\ so} \\
        \magn(C_j',i_2,k_2) &= \mint(C_j',k_2)>\mint(C_j,k_2)
        \enspace.
    \end{align*}
\item None of the intensities decreases.
Indeed the only differences are due to changes to permutations
$\sigma\in C_j$ with $\sigma(i_1)=k_2$. Such a permutation where $\sigma(i_2)=k_3$, say, is replaced
by $\sigma'$, where $\sigma'(i_2)=k_2$ and $\sigma'(i_1)=k_3$, if $\sigma'$ is not already in $C_j$.
As shown in (ii), the intensity of $k_2$ increases. For $k_3$, only $\magn(C_j,i_2,k_3)$ decreases,
but since $i_2$ was not a magnet in $C_j$, the magnet in $C'_j$ of $k_3$,  and hence $\mint(C_j,k_3)$, is unchanged.
\end{enumerate}

We repeat this operation for every remaining pair of elements which share a magnet in $C_j$ until we arrive at a set of permutations which has $n$ distinct magnets. Then, we perform the same process for every other class in $\Strat$.

To see that this algorithm indeed terminates, (ii) shows that if in any iteration the magnet of an element $i$ changes, then $\mint(C_j',i)>\mint(C_j,i)$. As the maximum intensity of any element within a class $C_j$ is $|C_j|$ and the minimum is $1$, the algorithm terminates after 
$n\cdot n!$ iterations.
		
Let us consider the collection $C = \langle A_0, \dots, A_{n-1}\rangle$ obtained. From (i), we see that the sets of $C$ contain a total of $n!$ permutations of $\PER$. Permutations belonging to the same set $A_j$ are necessarily distinct, but two different sets of $C$ may have non-trivial intersection. Hence, $C$ may not be a strategy. Every $A_j$ has $n$ distinct magnets, one for each element of $\nset$. Most importantly, by (iii), we have
\begin{align*}
    \sum_{\substack{i, j \in \nset}} \mint(A_j,i) &\ge
    \sum_{\substack{i, j \in \nset}} \mint(C_j,i) =
    F(n)
    \enspace.
\end{align*}
Hence, calculating an upper bound for $\sum_{i, j \in \nset} \mint(A_j,i)$ yields an upper bound for $F(n)$.

The set $A_j$ has exactly $n$ magnets, one for each element of $\nset$. For a permutation $\sigma \in A_j$ to contribute $r$ to $\sum_{i \in \nset} \mint(A_j,i)$, $\sigma^{-1}$ must map exactly $r$ elements to their magnets in $A_j$. Hence, (see \Cref{def:derangements}) there are at most $D_{n,r}$ permutations in $A_j$ which contribute exactly $r$ to $\sum_{i \in \nset} \mint(A_j,i)$.
Recall that $D_{n,r} \le \frac{n!}{r!}$ and thus for any natural $\ell$,
\begin{align*}
    \sum_{i \in \nset} \mint(A_j,i) &\le
    \ell \cdot |A_j| + \sum_{r=\ell+1}^n r \cdot D_{n,r} =
    \ell \cdot |A_j| + \sum_{r=\ell+1}^n \frac{n!}{(r-1)!} \le
    \ell \cdot |A_j| + \frac{en!}{\ell!}
    \enspace.
\end{align*}
We will choose some $\ell = \frac{(1+o(1)) \log n}{\log\log n}$ to ensure that $\ell! = \omega(n)$, giving
\begin{align} 
    F(n) &\le \mkern-18mu
    \sum_{i, j \in \nset} \mkern-20mu\mint(A_j,i)
        \le \mkern-16mu
    \sum_{j \in \nset}  \mkern-15mu \left(\ell \cdot |A_j| + o((n-1)!)\right)
        =
    (\ell + o(1)) n!
        =
    \dfrac{(1+o(1)) \log n}{\log\log n}  n!
    \enspace.
    \label{ineq:7}
\end{align}

We can combine (\ref{upper-bound:prob}) and (\ref{ineq:7}) to obtain the following,
\begin{align*}
    \Pr{\mathcal{V}} &
        \le
    \frac1n \cdot \frac{F(n)}{n!}
        \le
    \frac{(1+o(1)) \log n}{n  \log\log n}
    \enspace.
\end{align*}

The upper bound of $\frac{(1+o(1))\log n}{n \log\log n}$ is valid not only for deterministic strategies, but also for \emph{randomized strategies}. Let $c(\Strat,(\sigma,i))$ be the indicator function of the event that the strategy $\Strat$ fails to guess the image of $i$ under the permutation $\sigma$. Let us consider a probability measure $P$ over the set $D$ of all deterministic strategies, and the distribution $U = (U_{\PER}, U_{\nset})$ over $\PER \times \nset$, where $U_S$ denotes the uniform probability measure over the set $S$. Let $S$ be a random strategy chosen according to $P$, and let $X$ be a random set-up chosen according to $U$. Then, by Yao's principle, $\max_{(\sigma,i) \in \PER \times \nset} \Ex{c(S,(\sigma,i))} \ge \min_{\Strat \in D} \Ex{c(\Strat,X)}$. That is, the probability that a randomized strategy fails for the worst-case input exceeds the probability that an optimal deterministic strategy fails. Hence, the worst-case probability that a randomized strategy succeeds is also bounded above by $\frac{(1+o(1))\log n}{n \log\log n}$.
\end{proof}


\section{Lower bound: solution for the needle in a haystack search}
\label{sec:lower}

In \Cref{thm:upper}, we showed that whatever strategy we use in the \emph{needle in a haystack} problem, the best success probability we can hope for is $\frac{(1+o(1)) \log n}{n  \log\log n}$. In this section we will show that such success probability is achievable by a simple strategy, which we call the \emph{shift strategy}.

\medskip\InGray{
\begin{itemize}
\item Let $\hint \in \nset$ maximize $|\{\ell \in \nset: \ell = \sigma(\ell + \hint \pmod n)\}|$.
\item In order to find number $\sought \in \nset$ in $\sigma$, check \textcolor[rgb]{1.00,0.00,0.00}{$\sigma(\sought + \hint \pmod n)$}.
\end{itemize}
}\medskip
\noindent (Observe that our choice of \hint is equivalent to maximizing $|\{\ell \in \nset: (\ell - \hint \pmod n) = \sigma(\ell) \}|$.)

\medskip

We will prove that the shift strategy ensures a success probability of at least $\frac{(1+o(1)) \log n}{n  \log\log n}$. Notice that this is equivalent to saying that $\Pr{\sigma(\sought + \hint \pmod n) = \sought} \ge \frac{(1+o(1)) \log n}{n  \log\log n}$, and hence, by the definition of $\hint$, that with probability $1-o(1)$,
\begin{align*}
    \max_{s \in \nset}\Big\{\big|\{\ell \in \nset: \ell - \sigma(\ell) = s \pmod n\}\big|\Big\} &\ge
    \frac{(1+o(1)) \log n}{\log\log n}
    \enspace.
\end{align*}
This also implies, by \Cref{thm:upper} (\Cref{sec:upper}), that the shift strategy is asymptotically optimal.

\begin{theorem}
\label{thm:lower}
For any $\sought \in \nset$, the shift strategy satisfies $\Pr{\mathcal{V}} \ge \frac{(1+o(1)) \log n}{n  \log\log n}$.
\end{theorem}

In order to prove \Cref{thm:lower}, we introduce some notation.
For every $i \in \nset$, let $v(i) = i - \sigma(i) \pmod n$. Since $\sigma$ is random, $v(i)$ has uniform distribution over $\nset$.

Let $S_{\ell} = |\{ i \in \nset: v(i) = \ell\}|$. Notice that in the shift strategy $\Strat = \langle C_0, C_1, \dots, C_{n-1} \rangle$, if $\sigma \in C_{\hint}$ then $S_{\hint} = \max_{\ell \in \nset} \{S_{\ell}\}$. Therefore, our goal is to study basic properties of the distribution of $S_{\hint}$, and in particular, to estimate the largest value of $S_j$ over all $j \in \nset$.

\begin{example}
\label{example-more-concrete}
Using the example presented in \Cref{fig:cards} with
\begin{align*}
  \sigma(0, 1, \dots, 51) =
	\langle
	49, 17, 1, 38, 27, 7, 21, 25, 45, 3, 51, 9, 35, 36, 11, 33, 23, 8, 46, 18, 13, 28, 26, 14, 2, 5, \\
	 10, 39, 48, 32, 29, 40, 19, 4, 12, 41, 50, 43, 6, 22, 34, 44, 24, 15, 16, 20, 0, 47, 30, 42, 31, 37
	\rangle ,
\end{align*}
we have
\begin{align*}
    v(0, 1, \dots,  51) =
        \langle
       3, 36, 1, 17, 29, 50, 37, 34, 15, 6, 11, 2, 29, 29, 3, 34, 45, 9, 24, 1, 7, 45, 48, 9, 22, 20, \\
         16, 40, 32, 49, 1, 43, 13, 29, 22, 46, 38, 46, 32, 17, 6, 49, 18, 28, 28, 25, 46, 0, 18, 7, 19, 14
         \rangle .
     \end{align*}
Then
\begin{align*}
    S_{0, 1, 2, \dots, 50, 51} = \ &
        \langle
        1, 3, 1, 2, 0, 0, 2, 2, 0, 2, 0, 1, 0, 1, 1, 1, 1, 2, 2, 1, 1, 0, 2, 0, 1, 1,\\
         & \; 0, 0, 2, 4, 0, 0, 2, 0, 2, 0, 1, 1, 1, 0, 1, 0, 0, 1, 0, 2, 3, 0, 1, 2, 1, 0
         \rangle ,
\end{align*}
so $\hint = 29$ and $S_{\hint} = 4$. Alice delivers this hint to Bob by exchanging cards $\heartsuit$5 and $\spadesuit$Q.
Then, over all $\sought \in \nset$, $\Pr{\sigma(\sought + 29 \pmod {52}) = \sought} = \frac{4}{52}$.
\hfill$\boxtimes$
\medskip
\end{example}

Let us first notice the following simple auxiliary lemma which should give the \emph{intuition} behind our approach 
(see \Cref{proof:lemma:expected} for a standard and elementary proof).

\begin{lemma}
\label{lemma:expected}
The expected number of values $j \in \nset$ with $S_j \ge \frac{(1+o(1)) \cdot \log n}{\log\log n}$ is at least one.
\end{lemma}


\Cref{lemma:expected} tells us that in expectation, there is at least one value $j$ such that $S_j \ge \frac{(1+o(1)) \log n}{\log\log n}$. Notice however that in principle, we could have that the expectation is high but only because with small probability the random variable takes a very high value. Therefore the bound in \Cref{lemma:expected} is fairly weak. We will now prove, using the second moment method, that with high probability there is some $j$ such that $S_j \ge \frac{(1+o(1)) \log n}{\log\log n}$. This yields \Cref{thm:lower}.

\begin{lemma}
\label{lemma:hp}
With probability $1 - o(1)$ there is some
$j \in \nset$ such that $S_j \ge \frac{(1+o(1)) \log n}{\log\log n}$.
\end{lemma}

\begin{proof}
Let $Z_j^t$ be the indicator random variable that $S_j = t$. Let $R_t = \sum_{j=0}^{n-1} Z_j^t$. With this notation, our goal is to show that $R_t = 0$ is unlikely for our choice of some $t = \frac{(1+o(1)) \log n}{\log\log n}$ (since if $R_t > 0$ then $\max_{j \in \nset} S_j \ge t$, and hence $\PPr{\max_{j \in \nset} S_j \ge t} \ge \Pr{R_t > 0}$). We use the second moment method relying on a standard  implication of Chebyshev's inequality,
\begin{align}
\label{ineq:Chebyshev}
    \PPr{\max_{j \in \nset} S_j < t} &\le
    \PPr{R_t = 0} \le
    \frac{\Var{R_t}}{\Ex{R_t}^2}
    \enspace.
\end{align}
Let us recall that
\begin{align}
\label{ineq:full-var}
    \Var{R_t} &=
    \Var{\sum_{j=0}^{n-1} Z_j^t} =
    \sum_{j=0}^{n-1} \Var{Z_j^t} +
    \sum_{i, j \in \nset, i \ne j} \Cov{Z_i^t, Z_j^t}
    \enspace.
\end{align}
Next, since every $Z_j^t$ is a 0-1 random variable, we obtain the following,
\begin{align}
\label{ineq:var}
    \Var{Z_j^t} &=
    \Pr{Z_j^t = 1} \cdot \Pr{Z_j^t = 0} \le
    \Pr{Z_j^t = 1} =
    \Ex{Z_j^t}
    \enspace.
\end{align}
Our main technical claim is that the covariance of random variables $Z_j^t$, $Z_i^t$ is small. Although the proof of \Cref{lemma:bound-for-Cov} is the \emph{main technical contribution} of this section,
for the clarity of the presentation,
we defer its proof to \Cref{sec:proof-of-lemma:bound-for-Cov}.

\begin{lemma}
\label{lemma:bound-for-Cov}
Let $t \le O(\log n)$. Then, the following holds for any $i \ne j$, $i, j \in \nset$:
\begin{align}
\label{ineq:bound-for-Cov}
    \Cov{Z_i^t, Z_j^t} &=
        \EEx{Z_i^t \cdot Z_j^t} - \EEx{Z_i^t} \cdot \Ex{Z_j^t}
        \le o(1) \cdot \EEx{Z_i^t} \cdot \Ex{Z_j^t}
    \enspace.
\end{align}
\end{lemma}

Therefore, if we combine (\ref{ineq:var}) and \Cref{lemma:bound-for-Cov} in identity (\ref{ineq:full-var}), then (assuming $t \le O(\log n)$)
\begin{align*}
    \Var{R_t} &=
    \sum_{j=0}^{n-1} \Var{Z_j^t} + \!\!\!\!
    \sum_{i, j \in \nset, i \ne j} \!\!\!\! \Cov{Z_i^t, Z_j^t} \le
    \sum_{j=0}^{n-1} \Ex{Z_j^t} +
    o(1)  \mkern-10mu \sum_{i, j \in \nset, i \ne j} \!\!\!\! \EEx{Z_i^t} \Ex{Z_j^t}
    \\&=
    \Ex{R_t} + o(1) \cdot \Ex{R_t}^2
    \enspace.
\end{align*}
If we plug this in (\ref{ineq:Chebyshev}), then we will get the following (assuming $t \le O(\log n)$),
\begin{align}
\label{ineq:balls-bins-2nd-method}
    \PPr{R_t  = 0} &\le
    \frac{\Var{R_t}}{\Ex{R_t}^2} \le
    \frac{1}{\Ex{R_t}}  + o(1)
    \enspace.
\end{align}

Therefore, if for some $\varsigma > 0$ we have $\Ex{R_t} \ge \varsigma$ (with $t \le O(\log n)$) then the bound above yields $\PPr{\max_{i \in \nset} S_i < t} \le \frac{1}{\varsigma} + o(1)$. Hence we can combine this with (\ref{ineq:prob-Sj}) to obtain $\Ex{R_t} = \sum_{j=0}^{n-1} \Ex{Z_j^t} = \sum_{j=0}^{n-1} \Pr{S_j = t} > \frac{n}{2et!}$, which is $\omega(1)$ for any $t$ such that $t! = o(n)$. This in particular holds for some $t = \frac{(1+o(1)) \log n}{\log\log n}$, and thus concludes \Cref{lemma:hp}.
\end{proof}

\begin{remark}
\label{remark:balls-and-bins}
A reader may notice a close similarity of the problem of estimating $\max_{i \in \nset} S_i$ to the maximum load problem for balls and bins, where one allocates $n$ balls into $n$ bins i.u.r.
Indeed, random variables $S_0, \dots, S_{n-1}$ have similar distribution to the random variables $B_0, \dots, B_{n-1}$, where $B_i$ represents the number of balls allocated to bin $i$. However, the standard approaches used in the analysis of balls-and-bins processes seem to be more complicated in our setting. The main reason is that while every single random variable $S_i$ has approximately Poisson distribution with mean 1, as has $B_i$ too, the analysis of $\max_{i \in \nset} S_i$ is more complicated than the analysis of $\max_{i \in \nset} B_i$ because of the intricate \emph{correlation} of random variables $S_0, \dots, S_{n-1}$. For example, one standard approach to show that $\max_{i \in \nset} B_i \ge \frac{(1+o(1)) \log n}{\log\log n}$ with high probability relies on the fact that the load of a set of bins $B_i$ with $i \in I$ decreases if we increase the load of bins $B_j$ with $j \in J$, $I \cap I = \emptyset$. However, the same property holds only \emph{approximately} for $S_0, \dots, S_{n-1}$ (and in fact, the $o(1)$ error term in \Cref{lemma:bound-for-Cov} corresponds to this notion of ``approximately''; for balls and bins the covariance is known to be always non-positive). To see the difficulty (see also the classic reference for permutations \cite[Chapters~7--8]{Rio58}), notice that, for example, if $\sigma(i) = i + \ell$ then we cannot have $\sigma(i+1) = i + \ell$, meaning that there is a special correlation between $S_{\ell}$ (which counts $i$ with $\sigma(i) = i + \ell$) and $S_{\ell-1}$ (which counts $i$ with $\sigma(i+1) = i + \ell$). In particular, from what we can see, random variables $S_0, \dots, S_{n-1}$ are not negatively associated \cite{DR98}. In a similar way, we do not expect the Poisson approximation framework from \cite{ACMR95} (see also \cite[Chapter~5.4]{MU17}) to work here.~Our approach is therefore closer to the standard second moment method, see, e.g., \cite[Chapter~3]{Bol01} and~\cite{RS98}.
\end{remark}


\junk{
\subsection{Elementary proof of \Cref{lemma:expected}}

Let us recall \Cref{def:derangements} for derangements and $r$-partial derangements. The probability that a random permutation in $\PER$ is a derangement is $D_n/n! = \lfloor{\frac{n!}{e}+\frac12}\rfloor/{n!} \sim \frac{1}{e}$.
Let $u(n) = \lfloor \frac{n!}{e} + \frac12 \rfloor / \frac{n!}{e}$ and note that $D_n = u(n) \, n!/e$, that $u(n) = 1+o(1)$, and $u(n) > 0.9$ for all $n>1$. Since the permutation $\sigma \in \PER$ is chosen i.u.r.,
\begin{align*}
    \Pr{S_0 = k} &= \frac{D_{n,k}}{n!} = \frac{\binom{n}{k}D_{n-k}}{n!} = \frac{\binom{n}{k}\frac{(n-k)!}{e}u(n-k)}{n!}  = \frac{u(n-k)}{ek!}
    \enspace.
\end{align*}

The same bound can be obtained for $S_j$ for every $j \ge 0$. For any permutation $\sigma \in \PER$ and any integer $\ell \in \nset$, define permutation $\sigma_{\ell} \in \PER$ such that
$\sigma_{\ell}(i) = \sigma(i) + \ell \pmod n$.
For any permutation $\sigma \in \PER$ and any $\ell$, the operator $\sigma \mapsto \sigma_{\ell}$ is a bijection from $\PER$ to $\PER$, and a permutation $\sigma \in \PER$ with $\ell \in \nset$ has exactly $k$ fixed points if and only if permutation $\sigma_{\ell}$ has exactly $k$ points with $\sigma_{\ell}(i) = i + \ell \pmod n$. Hence for every $j, j' \in \nset$ and $k \in [n]$, we have $\Pr{S_j = k} = \Pr{S_{j'} = k}$.
Therefore, for any integers $j \in \nset$ and $k \in [n-2]$,
\begin{align}
\label{ineq:prob-Sj}
    \Pr{S_j = k} &=
       \frac{u(n-k)}{ek!} > \frac{1}{2ek!}
    \enspace.
\end{align}
Let $k(n)$ be the largest $k$ such that $2ek! \le n$. Then $\Pr{S_j = k(n)} > 1/n$. Hence, if we let $Q_j$ be the indicator random variable that $S_j = k(n)$, then $\Pr{Q_j=1} > 1/n$, and hence $\Ex{\sum_{j=0}^{n-1} Q_j} = \sum_{j=0}^{n-1} \Ex{Q_j} = \sum_{j=0}^{n-1} \Pr{Q_j=1} > 1$. Therefore, in expectation, there is at least one value $j$ such that $S_j = k(n)$. It is easy to show that $k(n) = \frac{\log n}{\log\log n} (1+o(1))$.
\qed
}


\section{Proof of \Cref{lemma:bound-for-Cov}: bounding the covariance of $Z_i^t$ and $Z_j^t$}
\label{sec:proof-of-lemma:bound-for-Cov}

The main technical part of the analysis of the lower bound for the \emph{needle in a haystack} search problem in \Cref{sec:lower} (see \Cref{thm:lower}) relies on the proof \Cref{lemma:hp}. This proof, in turn, is quite simple except for one central claim, \Cref{lemma:bound-for-Cov}, bounding the covariance of $Z_i^t$ and $Z_j^t$. The proof of \Cref{lemma:bound-for-Cov} is rather lengthly, and therefore for the convenience of the reader the proofs of some lemmas are deferred to \Cref{sec:appendix}.

Let $Z_j^t$ be the indicator random variable that $S_j = t$. Since $Z_i^t$ and $Z_j^t$ are 0-1 random variables, we have $\EEx{Z_i^t \cdot Z_j^t} = \Pr{S_i = t, S_j = t}$, $\EEx{Z_i^t} = \Pr{S_i = t}$ and $\Ex{Z_j^t} = \Pr{S_j = t}$. Since $\Pr{S_i = t} = \Pr{S_j = t} = \frac{u(n-t)}{et!} = \frac{1+o(1)}{et!}$ by (\ref{ineq:prob-Sj}), to complete the proof of
\Cref{lemma:bound-for-Cov}, we only have to show that, for $i \ne j$,
\begin{align}
\label{ineq:bound-for-Pr:S_i=S_j=t}
\Pr{S_i = t, S_j = t} &\le
(1 + o(1)) \cdot \frac{1}{(et!)^2}
\enspace.
\end{align}
We will prove this claim in \Cref{lemma:bound-for-Pr:S_i=S_j=t} in Section \ref{subsubsec:bound-for-Pr:S_i=S_j=t} below.

\subsection{Notation and key intuitions}

For any set $I \subseteq \nset$ and any integer $\ell \in \nset$, let $\fixed_{I,\ell} = \{\sigma \in \PER: \sigma(i) = i + \ell \pmod n \text{ iff } i \in I\}$ and $\fixed^*_{I,\ell} = \{\sigma \in \PER: \forall_{i \in I} \ \sigma(i) = i + \ell \pmod n\}$. Notice that $\fixed_{I,\ell} \subseteq \fixed^*_{I,\ell}$. Further, $|\fixed_{I,\ell}| = D_{n-t,0}$ where $t=|I|$, and
\begin{align*}
\Pr{S_i = t} &=
\frac{|\bigcup_{I \subseteq \nset, |I| = t} \fixed_{I,i}|}{n!} =
\frac{\sum_{I \subseteq \nset, |I| = t} |\fixed_{I,i}|}{n!} =
\frac{\binom{n}{t} \cdot D_{n-t,0}}{n!}
\enspace.
\end{align*}
Next, with this notation and for $i \ne j$, we also have
\begin{align*}
\Pr{S_i = t, S_j = t} &=
\frac{1}{n!}\left|\bigcup_{I, J \subseteq \nset, |I| = |J| = t}\mkern-20mu \fixed_{I,i} \cap \fixed_{J,j}\right|=
\frac{1}{n!}\sum_{I, J \subseteq \nset, |I| = |J| = t}\mkern-20mu |\fixed_{I,i} \cap \fixed_{J,j}|
\enspace.
\end{align*}
Notice that in the sum above one can restrict attention only to $I \cap J = \emptyset$, since $\fixed_{I,i} \cap \fixed_{J,j} = \emptyset$ otherwise. In view of this, our goal is to estimate $|\fixed_{I,i} \cap \fixed_{J,j}|$ for disjoint sets $I, J \subseteq \nset$.

In what follows, we will consider sets $S_i$ and $S_j$ with $i = 0$ and $j = s$ for some $s \in \nset \setminus \{0\}$. By symmetry, we can consider the first shift to be~0 without loss of generality; $s$ is an arbitrary non-zero value. As required in our analysis (see \Cref{lemma:bound-for-Cov}), we will consider $t \le O(\log n)$.

Our approach now is to focus on a typical pair $I$ and $J$, and consider some atypical pairs separately. We will show in \Cref{lemma:compatible-almost-all} that almost all pairs of disjoint sets $I$ and $J$ are so-called \emph{compatible for shift $s$}. As a result, the contribution of pairs $I$ and $J$ that are not compatible for $s$ is negligible, and so we will focus solely on pairs compatible for $s$. Then, for the pair of indices $I$ and $J$ we will estimate $|\fixed_{I,i} \cap \fixed_{J,j}|$ using the Principle of Inclusion-Exclusion. For that, we will have to consider the contributions of all possible sets $K \subseteq \nset \setminus (I \cup J)$ to the set of permutations in $\fixed_{I,i}^* \cap \fixed^*_{J,j}$. As before, contributions of some sets are difficult to be captured and so we will show in \Cref{lemma:feasible-almost-all} that almost all sets $K \subseteq \nset \setminus (I \cup J)$ are so-called \emph{feasible for $I$, $J$, and $s$}. As a result, the contribution of sets $K$ that are not feasible for $I$, $J$, and $s$ is negligible, and so we will focus on sets that are feasible for $I$, $J$, and $s$. The final simplification follows from the fact that we do not have to consider all such sets $K$, but only small sets $K$, of size $O(\log n)$. Once we have prepared our framework, we will be able to use the Principle of Inclusion-Exclusion to estimate $|\bigcup_{I, J \subseteq \nset, |I| = |J| = t} \fixed_{I,i} \cap \fixed_{J,j}|$ in \Cref{lemma:size-for-compatible-pairs,lemma:bound-for-Pr:S_i=S_j=t}.


\subsection{The analysis}

For any integer $\ell$ and any subset $L \subseteq \nset$ we write $L + \ell$ to denote the set of elements in $L$ shifted by $\ell$, in the arithmetic modulo $n$, that is, $L + \ell = \{i + \ell \pmod n: i \in L\}$. Similarly, $L - \ell = \{i - \ell \pmod n: i \in L\}$.

Let $\set_{0,s}(I,J) = \fixed_{I,0} \cap \fixed_{J,s} = \{\sigma \in \PER: \sigma(i) = i \text{ iff } i \in I \text{ and } \sigma(j) = j + s \pmod n \text{ iff } j \in J\}$.
Let $\set^*_{0,s}(I,J) = \fixed^*_{I,0} \cap \fixed^*_{J,s} = \{\sigma \in \PER: \forall_{i \in I} \ \sigma(i) = i \text{ and } \forall_{j \in J} \ \sigma(j) = j + s \pmod n\}$. 

It is easy to compute the size of $\set^*_{0,s}(I,J)$. Notice first that if $I \cap J \ne \emptyset$ or $I \cap (J+s) \ne \emptyset$, then $\set^*_{0,s}(I,J) = \set_{0,s}(I,J) = \emptyset$. Otherwise, if $I \cap J = \emptyset$ and $I \cap (J+s) = \emptyset$, then $|\set^*_{0,s}(I,J)| = (n - |I \cup J|)!$ (see also \ref{lemma:compatible-aux1}).

However, our main goal, that of computing the size of $\set_{0,s}(I,J)$, is significantly more complicated, because this quantity cannot be reduced to an intersection test and a simple formula  over $n$, $|I|$, $|J|$, and $s$. 

\subsubsection{Disjoint sets $I \subseteq \nset$ and $J \subseteq \nset \setminus I$ compatible for shift $s$}
\label{subsub:Disjoint-sets-compatible-for-shift}

Let $I$ and $J$ be two arbitrary subsets of $\nset$ of size $t$ each. We say \emph{$I$ and $J$ are compatible for shift $s$} if the four sets $I$, $J$, $I - s$, and $J + s$ are all pairwise disjoint. With this notation, we have the following lemma.

\begin{lemma}
	\label{lemma:compatible-aux1}
	If $I$ and $J$ are compatible for shift $s$, then $\set_{0,s}(I,J) \ne \emptyset$ and $|\set^*_{0,s}(I,J)| = (n - |I \cup J|)!$.
\end{lemma}

\begin{proof}
	If $I$ and $J$ are compatible for shift $s$ then any permutation $\sigma \in \PER$ with $\sigma(i) = i$ for all $i \in I$, $\sigma(j) = j + s \pmod n$ for all $j \in J$ and complemented by an arbitrary permutation $\nset \setminus (I \cup J)$ is in $\set^*_{0,s}(I,J)$. Hence the claim follows from the fact that since $I$, $J$, and $J+s$ are pairwise disjoint, such permutations always exist.
\end{proof}

The following lemma shows that almost all pairs of disjoint sets of size $t \le O(\log n)$ are compatible (see \ref{proof:lemma:compatible-almost-all} for a proof).

\begin{lemma}
	\label{lemma:compatible-almost-all}
	Let $s$ be an arbitrary non-zero integer in $\nset$. If we choose two disjoint sets $I, J \subseteq \nset$ of size $t$ i.u.r., then the probability that $I$ and $J$ are compatible for shift $s$ is at least $\left(1 - \frac{4t}{(n-2t)}\right)^{2t}$.
	
	In particular, if $t \le O(\log n)$, then this probability is at least $1 - O\left(\frac{\log^2n}{n}\right)$.
\end{lemma}

Because of Lemma~\ref{lemma:compatible-almost-all}, our goal will be to compute the sizes of sets $\set_{0,s}(I,J)$ only for compatible sets $I$ and $J$. Next, for given disjoint sets $I$ and $J$ compatible for shift $s$, we will consider all sets $K \subseteq \nset \setminus (I \cup J)$ and argue about their contributions to $|\set^*_{0,s}(I,J)|$ using the Principle of Inclusion-Exclusion.

\subsubsection{Properties of sets $K \subseteq \nset$ feasible for $I$, $J$, and $s$}
\label{subsubsec:Properties-feasible-sets}

Define $\eset_{I,J,0,s}(K) = \{\sigma \in \set^*_{0,s}(I,J): \text{ for every $\ell \in K$, } \sigma(\ell) \in \{\ell, \ell + s \pmod n\}\}$. While it is difficult to study $\eset_{I,J,0,s}(K)$ for all sets $K \subseteq \nset \setminus (I \cup J)$, we will focus our attention only on subsets with some good properties. We call a \emph{set $K \subseteq \nset$ feasible for $I$, $J$, and $s$}, if $I$ and $J$ are compatible for shift $s$, $K \cap (K+s) = \emptyset$, and $K \cap (I \cup J \cup (I-s) \cup (J+s)) = \emptyset$.

To justify the definition of feasible sets, we begin with the following simple lemma (see \ref{proof:lemma:feasible-aux1} for a proof).

\begin{lemma}
\label{lemma:feasible-aux1}
If $K \subseteq \nset$ is feasible for $I$, $J$, and $s$, then $|\eset_{I,J,0,s}(K)| = 2^{|K|} \cdot (n - |I \cup J \cup K|)!$.
\end{lemma}

Next, similarly to Lemma~\ref{lemma:compatible-almost-all}, we argue that almost all suitably small sets are feasible for pairs of disjoint small sets (see \ref{proof:lemma:feasible-almost-all} for a simple proof).

\begin{lemma}
\label{lemma:feasible-almost-all}
Let $s$ be an arbitrary non-zero integer in $\nset$. Let $I$ and $J$ be a pair of compatible sets for $s$ with $|I| = |J| = t$. Let $k$ be a positive integer with $2k \le n-4t$. If we choose set $K \subseteq \nset \setminus (I \cup J)$ of size $k$ i.u.r., then the probability that $K$ is feasible for $I$, $J$, and $s$ is at least $\left(1 - \frac{2t+k}{n-2t-k}\right)^k$.
	
	In particular, if $t, k \le O(\log n)$, then this probability is at least $1 - O\left(\frac{\log^2n}{n}\right)$.
\end{lemma}

\subsubsection{Approximating $|\set_{0,s}(I,J)|$ for compatible sets $I, J$ for $s$}
\label{subsubsec:Approximating-set-for-compatible-sets}

In this section we will complete our analysis to provide a tight bound for the size of $\set_{0,s}(I,J)$ for any pair $I$ and $J$ of sets compatible for shift $s$ with $|I| = |J| \le O(\log n)$. Our analysis relies on properties of sets feasible for $I$, $J$, and $s$, as proven in \Cref{lemma:feasible-aux1,lemma:feasible-almost-all}.

We begin with the two auxiliary claims (for simple proofs, see \Cref{proof:claim:bound-for-feasible-K,proof:claim:bound-for-infeasible-K}). For both, let $r$ be the smallest integer such that $2r\ge \log_2 n$ and let $t=|I| = |J| \le O(\log n)$.

\begin{claim}
\label{claim:bound-for-feasible-K}
	\begin{align}
	\label{ineq:bound-for-feasible-K}
	\sum_{k=1}^{2r} (-1)^{k+1}
	\mkern-35mu
	\sum_{\substack{K \subseteq \nset \setminus (I \cup J), |K|= k \\ \text{$K$ feasible for $I$, $J$, and $s$}}}
	\mkern-30mu
	|\eset_{I,J,0,s}(K)|
	&\ge
	\left(1 - O\left(\frac{\log^2n}{n}\right)\right) \cdot (n - 2t)! \cdot (1 - e^{-2})
	\enspace.
	\end{align}
\end{claim}

\begin{claim}
	\label{claim:bound-for-infeasible-K}
	\begin{align*}
	\sum_{k=1}^{2r} (-1)^{k+1}
	\sum_{\substack{K \subseteq \nset \setminus (I \cup J), |K|= k \\ \text{$K$ \emph{not} feasible for $I$, $J$, and $s$}}}
	|\eset_{I,J,0,s}(K)|
	&\ge
	- O\left(\frac{\log^2n}{n}\right) \cdot (n-2t)!
	\enspace.
	\end{align*}
\end{claim}

In order to approximate the size of $\set_{0,s}(I,J)$ for sets $I$ and $J$ compatible for shift $s$, let us first notice that
\begin{align}
\label{ineq:aux-2}
\set_{0,s}(I,J) &=
\set^*_{0,s}(I,J) \setminus \mkern-20mu \bigcup_{\ell \in \nset \setminus (I \cup J)} \mkern-20mu \eset_{I,J,0,s}(\{\ell\})
\enspace.
\end{align}
Therefore, since we know that $|\set^*_{0,s}(I,J)| = (n-(|I|+|J|))!$ by (\ref{lemma:compatible-aux1}), we only have to approximate $|\bigcup_{\ell \in \nset \setminus (I \cup J)} \eset_{I,J,0,s}(\{\ell\})|$; we need a good lower bound.

We compute $|\bigcup_{\ell \in \nset \setminus (I \cup J)} \eset_{I,J,0,s}(\{\ell\})|$ using the Principle of Inclusion-Exclusion, 
\begin{align*}
| \mkern-20mu \bigcup_{\ell \in \nset \setminus (I \cup J)} \mkern-20mu \eset_{I,J,0,s}(\{\ell\})| &=
\sum_{K \subseteq \nset \setminus (I \cup J), K \ne \emptyset} \mkern-20mu (-1)^{|K|+1} |\bigcap_{\ell \in K}  \eset_{I,J,0,s}(\{\ell\})| \\&=
\sum_{K \subseteq \nset \setminus (I \cup J), K \ne \emptyset} \mkern-20mu (-1)^{|K|+1} |\eset_{I,J,0,s}(K)| \\&=
\sum_{k=1}^{n-(|I|+|J|)} (-1)^{k+1} \mkern-20mu \sum_{K \subseteq \nset \setminus (I \cup J), |K|= k}\mkern-20mu  |\eset_{I,J,0,s}(K)|
\enspace.
\end{align*}

We will make further simplifications; since computing $|\eset_{I,J,0,s}(K)|$ for arbitrary non-empty sets $K \subseteq \nset \setminus (I \cup J)$ is difficult, we restrict our attention only to \emph{small} sets $K$ which are \emph{feasible for $I$, $J$, and~$s$}. For that, we will need to show that by restricting only to small sets $K$
feasible for $I$, $J$, and $s$, we will not make too big errors in the calculations.

Let $r$ be the smallest integer such that $2r\ge \log_2 n $. We can use the Bonferroni inequality \cite{Bon36} to obtain the following,
\begin{align}
\lefteqn{
	| \mkern-20mu \bigcup_{ \mkern20mu \ell \in \nset \setminus (I \cup J) \mkern-20mu }  \mkern-20mu \eset_{I,J,0,s}(\{\ell\})| \quad\ge\quad
	\sum_{k=1}^{2r} (-1)^{k+1}  \mkern-20mu \sum_{K \subseteq \nset \setminus (I \cup J), |K|= k} |\eset_{I,J,0,s}(K)|
}
\nonumber\\&\qquad\qquad=
\sum_{k=1}^{2r} (-1)^{k+1} \Bigg(\sum_{\substack{K \subseteq \nset \setminus (I \cup J), |K|= k \\ \text{$K$ feasible for $I$, $J$, and $s$}}}  \mkern-40mu |\eset_{I,J,0,s}(K)| +  \mkern-40mu \sum_{\substack{K \subseteq \nset \setminus (I \cup J), |K|= k \\ \text{$K$ \emph{not} feasible for $I$, $J$, and $s$}}}  \mkern-50mu |\eset_{I,J,0,s}(K)|\Bigg)
\nonumber\\&\qquad\qquad\ge
- O\left(\frac{\log^2n}{n}\right) (n-2t)! +
\left(1 - O\left(\frac{\log^2n}{n}\right)\right)  (n - 2t)! \cdot (1 - e^{-2})
\nonumber\\&\qquad\qquad=
\left(1 - O\left(\frac{\log^2n}{n}\right)\right)  (n - 2t)! \cdot (1 - e^{-2})
\label{ineq:aux-3}
\enspace,
\end{align}
where the last inequality follows from the auxiliary Claims \ref{claim:bound-for-feasible-K} and \ref{claim:bound-for-infeasible-K}.

If we combine (\ref{ineq:aux-2}) and (\ref{ineq:aux-3}), then we get the following lemma.

\begin{lemma}
	\label{lemma:size-for-compatible-pairs}
	If $I$ and $J$ are compatible for shift $s$ and $|I| = |J| = t=O(\log n)$, then
	\begin{align*}
	|\set_{0,s}(I,J)| &=
	|\set^*_{0,s}(I,J)| - | \mkern-20mu \bigcup_{\ell \in \nset \setminus (I \cup J)}  \mkern-20mu \eset_{I,J,0,s}(\{\ell\})|
	\le
	\frac{(n-2t)!}{e^2}  \left(1 + O\left(\frac{\log^2n}{n}\right)\right)
	\enspace.
	\end{align*}
\end{lemma}
\junk{
	\begin{align*}
	|\set_{0,s}(I,J)| &=
	|\set^*_{0,s}(I,J)| - |\bigcup_{\ell \in \nset \setminus (I \cup J)} \eset_{I,J,0,s}(\{\ell\})|
	\\&\le
	(n-2t)! - \left(1 - O\left(\frac{\log^2n}{n}\right)\right) \cdot (n - 2t)! \cdot \left(1 - e^{-2} - \frac{1}{n}\right)
	\\&=
	(n - 2t)! \cdot \left(1 - (1 - e^{-2} - 1/n) + O(\log^2n/n) \cdot (1 - e^{-2} - 1/n)\right)
	\\&=
	(n - 2t)! \cdot \left(e^{-2} + 1/n + O(\log^2n/n) \cdot (1 - e^{-2} - 1/n)\right)
	\\&=
	(n - 2t)! \cdot (e^{-2} + O(\log^2n/n))
	\\&=
	\frac{(n-2t)!}{e^2} \cdot \left(1 + O\left(\frac{\log^2n}{n}\right)\right)
	\enspace.
	\end{align*}
}

\begin{proof}
	Indeed, by (\ref{ineq:aux-2}), we have
	\begin{align*}
	|\set_{0,s}(I,J)| = |\set^*_{0,s}(I,J)| - | \mkern-20mu \bigcup_{\ell \in \nset \setminus (I \cup J)}  \mkern-20mu \eset_{I,J,0,s}(\{\ell\})|
	\enspace,
	\end{align*}
	by Lemma~\ref{lemma:compatible-aux1} we get
	\begin{align*}
	|\set^*_{0,s}(I,J)| = (n-(|I|+|J|))!
	\enspace,
	\end{align*}
	and by (\ref{ineq:aux-3}) we have
	\begin{align*}
	| \mkern-20mu \bigcup_{\ell \in \nset \setminus (I \cup J)}  \mkern-20mu \eset_{I,J,0,s}(\{\ell\})\ |
	&\ge
	\left(1 - O\left(\frac{\log^2n}{n}\right)\right) \cdot (n - 2t)! \cdot (1 - e^{-2})
	\enspace.
	\end{align*}
	Putting these three bounds together yields the promised bound.
\end{proof}

\subsubsection{Completing the proof of inequality (\ref{ineq:bound-for-Pr:S_i=S_j=t})}
\label{subsubsec:bound-for-Pr:S_i=S_j=t}

Now, with (\ref{lemma:size-for-compatible-pairs}) at hand, we are ready to complete our analysis in the following lemma.

\begin{lemma}
\label{lemma:bound-for-Pr:S_i=S_j=t}
For any $i, j \in \nset$, $i \ne j$, and for $t \le O(\log n)$, we have,
\begin{align*}
	\Pr{S_i = t, S_j = t} &\le
	\left(1 + O\left(\frac{\log^2n}{n}\right)\right) \frac{1}{(et!)^2}
	\enspace.
\end{align*}
\end{lemma}

\begin{proof}
Without loss of generality we assume that $i = 0$ and $j \in \nset \setminus \{0\}$.
	
First, let us recall the following
\begin{align*}
	\sum_{I, J \subseteq \nset, |I| = |J| = t, I \cap J = \emptyset}\mkern-70mu  |\fixed_{I,0} \cap \fixed_{J,j}| &=
	\sum_{I, J \subseteq \nset, |I| = |J| = t, I \cap J = \emptyset} \mkern-70mu |\set_{0,j}(I,J)| \\
	&=    \sum_{\substack{I, J \subseteq \nset, |I| = |J| = t, I \cap J = \emptyset \\ \text{$I$ and $J$ \emph{not} compatible for $j$}}}
	\mkern-70mu |\set_{0,j}(I,J)|
	\ +    \sum_{\substack{I, J \subseteq \nset, |I| = |J| = t \\ \text{$I$ and $J$ compatible for $j$}}} \mkern-50mu |\set_{0,j}(I,J)|
	\enspace.
\end{align*}
	
Next, let us notice that if $I$ and $J$ are \emph{not} compatible for shift $j$ and $I \cap J = \emptyset$, then we clearly have $|\set_{0,s}(I,J)| \le (n-2t)!$ (since once we have fixed $2t$ positions, we can generate at most $(n-2t)!$ distinct $n$-permutations). Further, by (\ref{lemma:size-for-compatible-pairs}), we know that if $I$ and $J$ are compatible for shift $j$, 
then $|\set_{0,s}(I,J)| \le \frac{(n-2t)!}{e^2} \cdot \left(1 + O\left(\frac{\log^2n}{n}\right)\right)$. Next, we notice that by (\ref{lemma:compatible-almost-all}), we have,
	\begin{flalign*}
	|\{&I, J \subseteq \nset: |I| = |J| = t, I \cap J = \emptyset \text{ and $I$, $J$ \emph{not} compatible for $j$}\}| \\&=
	O\left(\dfrac{\log^2n}{n}\right) \big|\{I, J \subseteq \nset: |I| = |J| = t, I \cap J = \emptyset\}\big| =
	O\left(\dfrac{\log^2n}{n}\right) \binom{n}{t} \binom{n-t}{t}
	\enspace.
	\end{flalign*}
	This immediately gives,
	\begin{align*}
	\sum_{\substack{I, J \subseteq \nset, |I| = |J| = t, I \cap J = \emptyset \\ \text{$I$ and $J$ \emph{not} compatible for $j$}}}
	\mkern-70mu |\set_{0,j}(I,J)|
	&\le
	O\left(\dfrac{\log^2n}{n}\right) \binom{n}{t} \binom{n-t}{t} (n-2t)!
	=
	O\left(\dfrac{\log^2n}{n}\right) \frac{n!}{(t!)^2}
	\end{align*}
	and
	\begin{align*}
	\sum_{\substack{I, J \subseteq \nset, |I| = |J| = t \\ \text{$I$ and $J$ compatible for $j$}}} \mkern-50mu \set_{0,j}(I,J)|
	&\le
	\binom{n}{t} \binom{n-t}{t} \frac{(n-2t)!}{e^2} \left(1 + O\left(\frac{\log^2n}{n}\right)\right) \\
	&=   \left(1 + O\left(\dfrac{\log^2n}{n}\right)\right) \frac{n!}{(et!)^2}
	\enspace.
	\end{align*}
	
	Therefore,
	\begin{align*}
	\sum_{\substack{I, J \subseteq \nset , I \cap J = \emptyset \\ |I| = |J| = t }}\mkern-40mu |\fixed_{I,0} \cap \fixed_{J,j}|
	&=
	\sum_{\substack{I, J \subseteq \nset, |I| = |J| = t, I \cap J = \emptyset \\ \text{$I$ and $J$ \emph{not} compatible for $j$}}}
	\mkern-70mu |\set_{0,j}(I,J)| \quad +
	\sum_{\substack{I, J \subseteq \nset, |I| = |J| = t \\ \text{$I$ and $J$ compatible for $j$}}} \mkern-60mu |\set_{0,j}(I,J)|
	\\&\le
	\left(1 + O\left(\dfrac{\log^2n}{n}\right)\right) \frac{n!}{(et!)^2}
	\enspace.
	\end{align*}
	
	Hence, we can conclude that for $i \ne j$ we have,
	\begin{align*}
	\Pr{S_i = t, S_j = t} &=
	\frac{1}{n!}\sum_{\substack{I, J \subseteq \nset , I \cap J = \emptyset \\ |I| = |J| = t }} |\fixed_{I,i} \cap \fixed_{J,j}|
	\le
	\left(1 + O\left(\dfrac{\log^2n}{n}\right)\right) \cdot \frac{1}{(et!)^2}
	\enspace.
	\end{align*}
\end{proof}


\section{Analysis of the \emph{communication in the locker room} setting}
\label{sec:locker-room-analysis}

A lower bound for the success probability in the \LRP is provided by a straightforward adaptation of the \emph{shift strategy}: Alice enters her message relaying the most common shift \hint to locker 0, and Bob opens locker 0 and uses Alice's message to check location $(\sought+\hint) \bmod n$ for his card. This strategy ensures a success probability of $\frac{(1+o(1)) \log n)}{n \log\log n}$.

As in \Cref{sec:prelim,sec:upper}, we will consider the case when \sought is chosen i.u.r. from \nset (see \Cref{sec:setup-random-sought}).
In order to obtain an upper bound for the success chance in the \LRP, we shall introduce some intermediate settings, or ``protocols''. In the \emph{CLR} protocol Alice views the contents of all the lockers, interchanges the contents of two lockers, then Bob is given a number and can open two lockers in search of it (i.e., the \emph{CLR} protocol is the set of rules which govern the \LRP). In the \emph{NH} protocol Alice views the contents of all the lockers, communicates a message of length $\log n$ to Bob, then Bob is given a number and can open one locker in search of it (i.e., the \emph{NH} protocol is the set of rules which govern the \emph{needle in a haystack} game). Moreover, we can append the modifier ``-with-$r$-bits'' to NH, which substitutes $r$ for $\log n$ in the above description.

We write $\Pr{\mathcal{V}(\mathcal{P})}$ for the optimal probability of success in protocol $\mathcal{P}$ and
$\Pr{\mathcal{V}(\Strat,\mathcal{P})}$ for the probability of success for strategy $\Strat$ in protocol $\mathcal{P}$. For example, we have already shown that $\Pr{\mathcal{V}(NH)} = \frac{(1+o(1)) \log n}{n \log\log n}$.

\begin{lemma}
\label{locker lemma}
$\Pr{\mathcal{V}(\text{CLR})} \le \Pr{\mathcal{V}(\text{NH-with-$4\log n$-bits})}$.
\end{lemma}

\begin{proof}
We will interpolate between CLR and NH-with-$4\log n$-bits with two other protocols.
	
In the protocol \emph{CLR0}, Alice views the contents of all the lockers, interchanges the contents of two lockers, then Bob is given a number and can open two lockers in search of it, and he can recognize upon seeing the content of the first locker whether it has been altered by Alice.
	
In the protocol \emph{CLR1}, Alice views the contents of all the lockers, interchanges the contents of two lockers, leaves these two lockers open with their contents visible to Bob, then Bob is given a number and can open one locker in search of it.
	
Also, let \emph{Sim} be the strategy in \emph{NH-with-$4\log n$-bits} in which Alice uses her message to communicate to Bob the cards whose positions she would exchange, and the positions of these cards, if she encountered the permutation $\sigma$ while working in the \emph{CLR1} protocol, simulating an optimal strategy $\Strat$ in \emph{CLR1}. Since this is an ordered quadruple in $\nset^4$, it can indeed be communicated in at most $4\log n$ bits.
	
The proof is in four parts:
\begin{enumerate}[(i)]
\item $\Pr{\mathcal{V}(\textit{CLR})} \le \Pr{\mathcal{V}(\textit{CLR0})}$,
\item $\Pr{\mathcal{V}(\textit{CLR0})} \le \Pr{\mathcal{V}(\textit{CLR1})} + O(\frac1n)$,
\item $\Pr{\mathcal{V}(\textit{CLR1})} \le \Pr{\mathcal{V}(\textit{Sim, NH-with-4logn-bits})}$,
\item $\Pr{\mathcal{V}(\textit{Sim, NH-with-4logn-bits})} \le \Pr{\mathcal{V}(\textit{NH-with-4logn-bits})}$.
\end{enumerate}

(i), (iii), (iv) are straightforward and so we only have to show (ii). Let $p_t$ be the maximum probability that Bob finds his sought number in the $t^{th}$ locker that he opens, $t\in \{1,2\}$.
	
Firstly, we bound $p_1$. Suppose that Alice and Bob have settled on a specific strategy. Let $e_{x,w}$ be the probability that $\sigma$ is such that Alice's transposition sends the locker $w$ to card $x$. Evidently, $0 \le e_{x,w} \le \frac{n-1}{n}$ for all $x,w \in \nset$ and $\sum_{x, w \in \nset} e_{x,w} \le 2$.
	
Having received his number \sought, Bob has to open a specific locker, let us say $b = b(\sought)$. The probability that Bob happens upon the card \sought in the locker $b$ is at most $e_{\sought,b(\sought)}+\frac{1}{n}$ (either Alice substitutes the content of $b(\sought)$ for \sought, or the content of $b(\sought)$ is initially \sought and Alice does not interfere). Thus, choosing \sought i.u.r. from \nset, the probability that Bob finds \sought at his first try is at most $\frac{1}{n} (\sum_{\sought, b \in \nset} e_{\sought,b(\sought)} + \frac{1}{n}) < \frac{3}{n} = O(\frac{1}{n})$.
	
Then, we bound $p_2$. If Bob opens first one of the lockers whose contents have been altered by Alice, then there is one remaining locker for him to open, and he has at most as much information as in the \emph{CLR0} protocol. Hence, in this case, $p_2 \le \Pr{\mathcal{V}(\textit{CLR0})}$.
	
Alternatively, Bob first opens one of the lockers whose contents have not been altered by Alice. This requires a more detailed analysis of the \emph{CLR0} protocol.

Alice's choice of a transposition is informed solely by the initial permutation $\sigma$ of the cards inside the lockers. Hence, there should be a function $a: \PER \rightarrow \binom{\nset}{2}$ which directs Alice to a pair of lockers. Then, Bob's choice of a first locker to open is informed only by his sought number. Thus, there should be a function $b: \nset \rightarrow \nset$ which directs Bob to his first locker. Finally, Bob chooses his second locker by considering his sought number and the content of the first locker, so there should be a function $b': \{0,1\} \times \nset^2 \rightarrow \nset$ which directs Bob to his second locker (the binary factor distinguishes whether Bob's first locker has had its content altered by Alice or not). The strategy which Alice and Bob employ in the \emph{CLR0} protocol can therefore be identified with a triple $[a,b,b']$.

Let $E_{u,v}=a^{-1}(\{u,v\})$ be the event that Alice transposes the contents of the $u^{th}$ and $v^{th}$ lockers, and let $F_w=b^{-1}(w)$ be the event that Bob opens the $w^{th}$ locker first. Let $s(y,w) \subseteq \PER$ be the permutations which map $w$ to $y$, and let $G_y$ be the event that the initial content of Bob's first locker is $y$. Notice that $\Pr{E_{u,v}|F_w \cap G_y} = \frac{|a^{-1}(\{u,v\}) \cap s(y,w)|}{(n-1)!}$, $\Pr{F_w} = \frac{|b^{-1}(w)|}{n}$, and $\Pr{G_y} = \frac{1}{n}$.
Then, the probability that Bob finds his sought number in his second attempt given that his first locker was not altered by Alice is
\begin{align*}
    p_2 &\le
    \sum_{\substack{u, v, w, y \in \nset\\  u,v,w \text{\ distinct }}} \Pr{E_{u,v}|F_w\cap G_y} \cdot \Pr{F_w} \cdot \Pr{G_y} \cdot \Pr{\mathcal{V}(\textit{CLR0}) |E_{u,v} \cap F_w \cap G_y}
    \enspace.
\end{align*}
Observe that
\begin{align*}
    \Pr{\mathcal{V}(\textit{CLR0})|E_{u,v} \cap F_w\cap G_y}
        \le
    \frac{(n-2)!}{\left|\left(\PER \setminus \bigcup_{\ell \in \nset} a^{-1}(\{w,\ell\})\right)\cap s(y,w)\right|}+\frac{2}{n}
\enspace.
\end{align*}

This holds because, barring the $\frac{2}{n}$ probability for Bob's sought card to be in a locker whose content was changed by Alice, Bob is only going to find his sought card in his second locker if the permutation $\sigma$ maps both $w$ to $y$ and Bob's second locker to his sought card. There are exactly $(n-2)!$ such permutations, which yields the numerator. For the denominator, when Bob opens the locker $w$ and views the card inside, he sees that its content is $y$ and that it has not been touched by Alice, so he knows that $\sigma$ is a permutation which maps $w$ to $y$ and which does not prompt Alice to transpose $y$ with some other card, and there are exactly $\left|\left(\PER \setminus \bigcup_{\ell \in \nset} a^{-1}(\{w,\ell\})\right)\cap s(y,w)\right|$ such permutations.

Also, note that
\begin{align*}
    \bigcup_{\substack{u,v\in\nset \\ u,v,w \text{\ distinct}}}
     (a^{-1}(\{u,v\}) \cap s(y,w))
        &=
        (\PER \setminus \bigcup_{\ell \in \nset} a^{-1}(\{w,\ell\})) \cap s(y,w)
    \Rightarrow
        \\
    \sum_{\substack{u,v\in\nset \\ u,v,w \text{\ distinct}}} |a^{-1}(\{u,v\}) \cap s(y,w)|
        &=
    |(\PER \setminus \bigcup_{\ell \in \nset} a^{-1}(\{w,\ell\}))\cap s(y,w)|
    \enspace.
\end{align*}
Combining the above, we obtain that
\begin{align*}
    p_2 \le \mkern-25mu
    \sum_{\substack{u, v, w, y \in \nset\\  u,v,w \text{\ distinct }}}
        \mkern-10mu
        \frac{1}{n} \cdot
        \frac{|a^{-1}(\{u,v\}) \cap s(y,w)|}{(n-1)!} \cdot
        \frac{|b^{-1}(w)|}{n} \cdot
        \left(\Pr{\mathcal{V}(\textit{CLR0})|E_{u,v} \cap F_w \cap G_y}+\frac{2}{n}\right)
        \\[-1ex]
        \le
    \sum_{w, y \in \nset} \frac{1}{n} \cdot \frac{1}{n-1} \cdot \frac{|b^{-1}(w)|}{n}+\frac{2}{n}
        =
    \frac{1}{n-1}+\frac{2}{n}
    \enspace.
\end{align*}
Thus, in this case, $p_2 \le \frac{4}{n}$.

Ultimately, $p_2 \le \Pr{\mathcal{V}(\textit{CLR1})}+\frac{4}{n}$, and hence $\Pr{\mathcal{V}(\textit{CLR0})}\le p_1+p_2 \le \Pr{\mathcal{V}(\textit{CLR1})}+O(\frac{1}{n})$, concluding the proof.
\end{proof}

\begin{theorem}
\label{Alice-Bob-vs-needle-haystack}
$\Pr{\mathcal{V}(\textit{CLR})} \le \frac{(4+o(1)) \log n}{n \log\log n}$.
\end{theorem}

\begin{proof}
We use \Cref{locker lemma} along with the fact that $\Pr{\mathcal{V}(\textit{NH-with-4logn-bits})} \le \frac{(4+o(1)) \log n}{n \log\log n}$, which can be immediately derived from \Cref{thm:upper-general} in \Cref{subsec:generalizations-longer-message} by setting $m=n^4$.
\end{proof}


\section{Generalizations}
\label{sec:symmetric-game}

There are several natural generalizations of the problem studied in this paper and related questions about properties of random permutations,
which we will discuss here.

\subsection{Simple generalization: longer message}
\label{subsec:generalizations-longer-message}

In the \emph{needle in a haystack} problem, when Alice sends the message $\hint$ to Bob, there is no reason why she must choose a number in $\nset$; instead, she could transmit a number $\hint \in [m-1]$ for an arbitrary integer $m$. One can easily generalize the analysis from \Cref{thm:upper,thm:lower} in this setting for a large range of $m$.

Let us denote the maximum attainable sum of intensities received from partitioning $\PER$ to $m$ parts the \emph{$m$-field} of $\PER$, and denote it by $F(n,m)$. Fields are simply diagonal $m$-fields (fields of the form $F(n,n)$).

We have $F(n,1) = n!$ (yielding a success probability of $\frac{1}{n}$, corresponding to not receiving advice) and $F(n,m) = n \cdot n!$ for every $m \ge n!$ (yielding a success probability of $1$, corresponding to obtaining full information). For other values of $m$ we can follow the approach used in \Cref{thm:upper}. First, notice that there is $\ell = \frac{(1+o(1)) \log m}{\log\log m}$, such that $\frac{m}{\ell!} = o(1)$. Then, using the techniques from the proof of \Cref{thm:upper}, we obtain
\begin{align*}
F(n,m) &\le
\sum_{i \in \nset, j \in [m-1]} \mint(A_j,i) \le
\sum_{j \in [m-1]} \left(
\ell \cdot |A_j| + \sum_{r=\ell+1}^n r \cdot D_{n,r}
\right) \\&\le
\sum_{j \in [m-1]} \left(
\ell \cdot |A_j| + \frac{(1+o(1))n!}{\ell!}
\right) \le
n! \cdot \left(\ell + \frac{(1+o(1))m}{\ell!}\right) \\&\le
\ell \cdot n! \cdot (1+o(1)) =
\frac{(1+o(1)) \log m}{\log\log m} \cdot n!
\enspace.
\end{align*}
By (\ref{upper-bound:prob}), this yields the success probability of $\frac{(1+o(1)) \log m}{n  \log\log m}$, giving the following theorem.
\begin{theorem}
	\label{thm:upper-general}
	If Alice can choose a number $\hint \in [m]$, then the maximum attainable success probability is at most $\frac{(1+o(1)) \log m}{n  \log\log m}$. In particular, if $m = \text{poly}(n)$, then the maximum attainable success probability is at most $O\left(\frac{\log n}{n  \log\log n}\right)$.
\end{theorem}

Observe that \Cref{thm:upper-general} implies that, since for the algorithm presented in \Cref{thm:lower}, that is, one using the shift strategy with hint $\hint \in [n]$, the success probability is already $\Omega\left(\frac{\log n}{n  \log\log n}\right)$, the shift strategy is asymptotically optimal to within a constant factor for any hint \hint polynomial in $n$. A similar conclusion holds also for the communication in the locker room setting: even if Alice leaves Bob a message by altering the contents of a constant number $c$ of lockers rather than just one, this message is $c\log n$ bits long, and hence the success probability is still at most $O(\frac{\log n}{n\log \log n})$.

Asymptotic results for several other interesting domains of $m$ could be found in a similar way. However, for super-polynomial domains, the upper bound derived in the above manner is far away from the lower bound that we currently can provide in \Cref{thm:lower}. Determining some properties of the rate of growth of $F(n,m)$ for fixed $n$ would be a good step towards determining its values. With this in mind, we have the following natural conjecture.

\begin{conjecture}
	For any fixed $n$, the function $f(m) = F(n,m)$ is concave.
\end{conjecture}

\subsection{Optimal strategies}

Although we have successfully calculated the maximum field and the maximum success probability for the \emph{needle in a haystack} problem, the problem of determining a characterization of, or at least some major properties for, optimal strategies remains. Indeed, the only optimal strategy that we have explicitly described so far is the shift strategy (which is in fact a set of different strategies, since, for permutations which have several $S_{\hint}$'s of maximum size, there are multiple legitimate options for their class). A natural generalization of shift strategies are \emph{latin strategies}; in these, Alice and Bob decide on a $n\times n$ latin square $S$, and Alice's message indicates the row of $S$ which coincides with $\sigma$ at the maximum number of places.

We present a couple of interesting questions concerning the optimal strategies for $\PER$ in \emph{needle in a haystack}.

\begin{conjecture}
	For every natural number $n$, there is an optimal strategy for $\PER$ whose parts all contain exactly $(n-1)!$ permutations.
\end{conjecture}

\begin{conjecture}
	Optimal strategies are exactly latin strategies.
\end{conjecture}

\subsection{Alice-In-Chains}

Let us explore another specific strategy. The \emph{naive strategy} is to group permutations according to the content of location $0$. That is, $\sigma, \sigma'$ belong to the same class if and only if $\sigma(0) = \sigma'(0)$. This is a natural strategy to conceive, and it agrees with the common (erroneous) notion that efficiency in the lockers game cannot be improved beyond $O(\frac{1}{n})$. Indeed, straightforward calculations yield a success probability of $\frac{2}{n}$ for the naive strategy in the \emph{needle in a haystack} problem.

Intuitive though it is, in the preceding sections we have proven the naive strategy to be suboptimal. In fact, the naive strategy fails to fully utilize the possibilities provided by the problem's framework. In this subsection, we show that, by introducing only a minor restriction to our problem, the naive strategy can indeed become optimal. This also demonstrates that strategic efficiency is very sensitive to changes in our assumptions about the \emph{needle in a haystack}.

Suppose that Alice and Bob face a challenge similar to the \emph{needle in a haystack}, but with a restriction: if $\Strat = \langle C_0, \dots, C_{n-1} \rangle$ is the agreed-upon strategy, then it must hold that
\begin{align*}
    \exists_{s \in \nset} \
    \forall_{i \in \nset} \
    \exists_{h \in \nset} \
    \forall_{\sigma \in C_h} \
        \sigma(i) \ne s
    \enspace,
\end{align*}
that is, ``there exists a needle $s$ such that for each location $i$ there is a corresponding message $h$ from Alice which would suffice to warn Bob that $s$ is not in $i$''. We call this the \emph{Alice-In-Chains} (AIC) variant.

\begin{theorem}
The naive strategy is optimal in Alice-In-Chains.
\end{theorem}

\begin{proof}
To begin with, it is easy to see that the naive strategy is allowed by the Alice-In-Chains rules. For instance, we can take $s=0$, in which case the message $h=1$ can inform Bob that $s$ is not in locker $0$, and the message $h=0$ can inform Bob that $s$ is not in locker $i$ for any $i>0$.

We proceed by induction. For $n \le 3$, it is easy to see that the naive strategy is optimal, even without the restriction.
    	
Suppose that it is optimal for $n \le N$. Let $\Strat = \langle C_0, \dots, C_N \rangle$ be an optimal strategy for $\ensuremath{\mathbb{S}_{N+1}}\xspace$ in the AIC variant. Without loss of generality, let $s=N$.
    	
Let $A_m$ be the subset of $\ensuremath{\mathbb{S}_{N+1}}\xspace$ which contains every permutation that maps $N$ to $m$. To bound the field $F_{AIC}(N+1)$, we will try to maximize the sum of the intensities produced by distributing the members of $A_m$ across the $N+1$ classes. That is, we partition each $A_m$ into a collection $C^{(m)}=[A_{m,0},\dots,A_{m,N}]$ which maximizes the sum $\sum_{0\le s,h\le N}\mint(A_{m,h},s)$. We claim that
\begin{align}
\label{ineq:8}
    F_{AIC}(N+1) &\le \sum_{m=0}^N\,\sum_{0\le s,h\le N}\mint(A_{m,h},s)
    \enspace.
\end{align}

To see that, observe that partitioning one set of permutations to several does not decrease the sum of the intensities. Indeed,
\begin{align*}
    \mint(C_h,s) &=
    \magn(C_h,s,\mmag(C_h,s)) =
    \sum_{m=0}^N \magn(A_{m,h},s,\mmag(C_h,s))
        \\&\le
    \sum_{m=0}^N \magn(A_{m,h},s,\mmag(A_{m,h},s)) =
    \sum_{m=0}^N \mint(A_{m,h},s)
    \enspace.
\end{align*}
Hence,
\begin{align*}
    F_{AIC}(N+1) &=
    \sum_{0\le s,h\le N} \mint(C_h,s) \le
    \sum_{m=0}^N\,\sum_{0\le s,h\le N}\mint(A_{m,h},s)
    \enspace.
\end{align*}
However, each $A_m$ is a copy of $\PER$, and one of its parts must be empty (because of the restriction of AIC, and the fact that all of the members of $A_m$ agree on the image of $N$). Therefore, $\sum_{0\le s,h\le N}\mint(A_{m,h},s)= F_{AIC}(N) \hspace{1mm}$ for all $m \in [N]$, and so (\ref{ineq:8}) yields
\begin{align}
\label{ineq:9}
    F_{AIC}(N+1)\le (N+1)F_{AIC}(N)
    \enspace.
\end{align}
From our inductive hypothesis, the naive strategy is an optimal strategy for $\PER$ in the AIC variant, so $F_{AIC}(N)=2N!$, which from (\ref{ineq:9}) implies $F_{AIC}(N+1)\le 2(N+1)!$. Since the yield of the naive strategy for $\ensuremath{\mathbb{S}_{N+1}}\xspace$ is exactly $2(N+1)!$, we have that the naive strategy is optimal for $\ensuremath{\mathbb{S}_{N+1}}\xspace$ in the AIC variant.
    	
\begin{remark}
The above implies that the Alice-In-Chains variant has a maximum attainable probability of $\frac{2}{n}$. It also proves an interesting result about the form of optimal strategies: every optimal strategy in the \emph{needle in a haystack} setting is such that every element $c\in\nset$ has an image which is present in all of the strategy's classes.
\end{remark}   	
\end{proof}


\junk{
	\subsection{Simple generalization: longer message}
	
	In the \emph{needle in a haystack} problem, when Alice sends the message $j$ to Bob, there is no reason why she must choose a number in $\nset$; instead, she could transmit a number $j \in [m-1]$ for an arbitrary integer $m$. One can easily generalize the analysis from \Cref{thm:upper,thm:lower} in this setting for a large range of $m$.
	
	Let us denote the maximum attainable sum of intensities received from partitioning $\PER$ to $m$ parts the \emph{$m$-field} of $\PER$, and denote it by $F(n,m)$. Fields are simply diagonal $m$-fields (fields of the form $F(n,n)$).
	
	We have $F(n,1) = n!$ (yielding a success probability of $\frac{1}{n}$, corresponding to not receiving advice) and $F(n,m) = n \cdot n!$ for every $m \ge n!$ (yielding a success probability of $1$, corresponding to obtaining full information). For other values of $m$ we can follow the approach used in \Cref{thm:upper}. First, notice that there is $\ell = \frac{(1+o(1)) \log m}{\log\log m}$, such that $\frac{m}{\ell!} = o(1)$. Then, using the techniques from the proof of \Cref{thm:upper}, we obtain
	\begin{align*}
	F(n,m) &\le
	\sum_{i \in \nset, j \in [m-1]} \mint(A_j,i) \le
	\sum_{j \in [m-1]} \left(
	\ell \cdot |A_j| + \sum_{k=\ell+1}^n k \cdot D_{n,k}
	\right) \\&\le
	\sum_{j \in [m-1]} \left(
	\ell \cdot |A_j| + \frac{(1+o(1))n!}{\ell!}
	\right) \le
	n! \cdot \left(\ell + \frac{(1+o(1))m}{\ell!}\right) \\&\le
	\ell \cdot n! \cdot (1+o(1)) =
	\frac{(1+o(1)) \log m}{\log\log m} \cdot n!
	\enspace.
	\end{align*}
	By (\ref{upper-bound:prob}), this yields the success probability of $\frac{(1+o(1)) \log m}{n  \log\log m}$, giving the following theorem.
	\begin{theorem}
		\label{thm:upper-general}
		If Alice can choose a number $j \in [m]$, then the maximum attainable success probability is at most $\frac{(1+o(1)) \log m}{n  \log\log m}$. In particular, if $m = \text{poly}(n)$, then the maximum attainable success probability is at most $O\left(\frac{\log n}{n  \log\log n}\right)$.
	\end{theorem}
	
	The above theorem also implies an interesting observation about the \emph{communication in the locker room} setting: even if Alice leaves Bob a message by altering the contents of a constant number $c$ of lockers rather than just one, this message is $c\log_2 n$ bits long, hence the success probability is still at most $O(\frac{\log n}{n\log \log n})$.
	
	Asymptotic results for several other interesting domains of $m$ could be found in this way. However, for super-polynomial domains, the upper bound derived in the above manner is far away from the lower bound that we currently can provide in \Cref{thm:lower}. Determining some properties of the rate of growth of $F(n,m)$ for fixed $n$ would be a good step towards determining its values. With this in mind, we have the following natural conjecture.

	\begin{conjecture}
		For any fixed $n$, the function $f(m) = F(n,m)$ is concave.
	\end{conjecture}
	
	\junk{Let us dub the locker problem that we have tackled so far the \emph{(original) Cards-In-Lockers} game.
		
		An interesting variation of the Cards-In-Lockers\Mike{That seems OK for a name. Now we have already briefly mentioned the symmetric  game earlier. Why ``symmetric''?}\George{I don't know. Anyone with a better name on mind is free to rename it.} game is the \emph{Symmetric Cards-In-Lockers game}. Here, Alice is disallowed from changing anything in $\sigma$; she can simply view it, then pass it on to Bob, concealed and unaltered. As she does this, Alice is allowed to whisper to Bob her message, which again consists of a single member of $\nset$. Then, Bob is given a certain number from $\nset$ and he is allowed to open a single locker. If he finds his assigned number inside, he and Alice win.
		
		As explained in the Preliminaries (2.1), the Symmetric Cards-In-Lockers game does not differ too much from the original Cards-In-Lockers game. Every strategy in the latter is only $O(\frac{1}{n})$ more efficient than its natural counterpart in the former. Hence, in our analysis so far, we have allowed them to be interchangeable. However, in a sense, the Symmetric Cards-In-Lockers leads to slightly more natural questions. For one, adjusting the derivation for $\Pr{\mathcal{V}}$ from the previous sections, we obtain $\Pr{\mathcal{V}}=\frac{1}{n}(1+\frac{1}{n!}\sum_{i, j \in \nset}\mint(C_j,i))$. So, in this case, we are interested in maximizing $\sum_{i, j \in \nset}\mint(C_j,i)$, without excluding the cases when the two indices match. This is indeed a very natural combinatorial question (though, to our knowledge, absent from literature). It also provides a good framework for exploring other variants and combinatorial tangents.
		
		\junk{Let us denote $F_S(n) = \max_S\sum_{i, j \in \nset}\mint(C_j,i)$ the symmetric field of $\PER$ and $P_{S,n}$ the maximum attainable probability. To what extend can we approximate $P_{S,n}$?}}
	
	\subsection{Optimal strategies}
	
	Although we have successfully calculated the maximum field and the maximum success probability for the \emph{needle in a haystack} problem, the problem of determining a characterization of, or at least some major properties for, optimal strategies remains. Indeed, the only optimal strategy that we have explicitly described so far is the shift strategy (which is in fact a set of different strategies, since, for permutations which have several $S_j$'s of maximum size, there are multiple legitimate options for their class).
	
	We present a couple of interesting questions concerning the optimal strategies for $\PER$ in \emph{needle in a haystack}.
	
	\begin{conjecture}
		For every natural number $n$, there is an optimal strategy for $\PER$ whose parts all contain exactly $(n-1)!$ permutations.
	\end{conjecture}
	
	\begin{conjecture}
		Every optimal strategy is a shift strategy.
	\end{conjecture}
	
	\subsection{Alice-In-Chains}
	
	Let us explore another specific strategy. The \emph{freshman's strategy} is to group permutations according to the content of location $0$. That is, $\sigma, \sigma'$ belong to the same class if and only if $\sigma(0) = \sigma'(0)$. This is a natural strategy to conceive, and it agrees with the common (erroneous) notion that efficiency in the lockers game cannot be improved beyond $O(\frac{1}{n})$. Indeed, straightforward calculations yield a success probability of $\frac{2}{n}$ for the freshman's strategy in the \emph{needle in a haystack} problem.
	
	Intuitive though it is, in the preceding sections we have proven the freshman's strategy to be suboptimal. In fact, the freshman's strategy fails to fully utilize the possibilities provided by the problem's framework. In this subsection, we show that, by introducing only a minor restriction to our problem, the freshman's strategy can indeed become optimal. This also demonstrates that strategic efficiency is very sensitive to changes in our assumptions about the \emph{needle in a haystack}.
	
	Suppose that Alice and Bob face a challenge similar to the \emph{needle in a haystack}, but with a restriction: if $\Strat = \langle C_0, \dots, C_{n-1} \rangle$ is the agreed-upon strategy, then it must hold that
	\begin{align*}
	\exists_{c \in \nset} \
	\forall_{k \in \nset} \
	\exists_{j \in \nset} \
	\forall_{\sigma \in C_j} \
	\sigma(c) \ne k
	\enspace,
	\end{align*}
	that is, ``each possible location of the needle $c$ is not suggested by any permutation in some class of the strategy $\Strat$.'' We call this the \emph{Alice-In-Chains} (AIC) variant.
	
	\begin{theorem}
		The freshman's strategy is optimal in Alice-In-Chains.
	\end{theorem}
	
	\begin{proof}
		To begin with, it is easy to see that the freshman's strategy is allowed by the Alice-In-Chains rules.
		
		We proceed by induction. For $n \le 3$, it is easy to see that the freshman's strategy is optimal, even without the restriction.
		
		Suppose that it is optimal for $n \le N$. Let $\Strat = \langle C_0, \dots, C_{n-1} \rangle$ be an optimal strategy for $S_{N+1}$ in the AIC variant. Without loss of generality, let $c=N$.
		
		Let $A_m$ be the subset of $S_{k+1}$ which contains every permutation that maps $N$ to $m$. To bound $F_{AIC}(N+1)$, we will try to maximize the sum of the intensities produced by distributing the members of $A_m$ across $N+1$ classes. That is, we partition each $A_m$ into a collection $C^{(m)}=[A_{m,0},\dots,A_{m,N}]$ which maximizes the sum $\sum_{0\le i,j\le N}\mint(A_{m,j},i)$. We claim that
		\begin{align}
		\label{ineq:8}
		F_{AIC}(N+1) &\le \sum_{m=0}^N\sum_{0\le i,j\le N}\mint(A_{m,j},i)
		\enspace.
		\end{align}
		
		To see that, observe that partitioning one set of permutations to several does not decrease the sum of the intensities. Indeed,
		\begin{align*}
		\mint(C_j,i) &=
		\magn(C_j,i,\mmag(C_j,i)) =
		\sum_{m=0}^N \magn(A_{m,j},i,\mmag(C_j,i))
		\\&\le
		\sum_{m=0}^N \magn(A_{m,j},i,\mmag(A_{m,j},i)) =
		\sum_{m=0}^N \mint(A_{m,j},i)
		\enspace.
		\end{align*}
		Hence,
		\begin{align*}
		F_{AIC}(N+1) &=
		\sum_{0\le i,j\le N} \mint(C_j,i) \le
		\sum_{m=0}^N\sum_{0\le i,j\le N}\mint(A_{m,j},i)
		\enspace.
		\end{align*}
		However, each $A_m$ is a copy of $S_N$, and one of its parts must be empty (because of the restriction of AIC, and the fact that all of the members of $A_m$ agree on the image of $N$). So, $\sum_{0\le i,j\le N}\mint(A_{m,j},i)= F_{AIC}(N) \hspace{1mm}\forall m \in [N]$, so that (\ref{ineq:8}) yields
		\begin{align}
		\label{ineq:9}
		F_{AIC}(N+1)\le (N+1)F_{AIC}(N)
		\enspace.
		\end{align}
		From our inductive hypothesis, the freshman's strategy is an optimal strategy for $S_N$ in the AIC variant, so $F_{AIC}(N)=2N!$, which from (\ref{ineq:9}) implies $F_{AIC}(N+1)\le 2(N+1)!$. Since the yield of the freshman's strategy for $S_{N+1}$ is exactly $2(N+1)!$, we have that the freshman's strategy is optimal for $S_{N+1}$ in the AIC variant.
		
		\begin{remark}
			The above implies that the Alice-In-Chains variant has a maximum attainable probability of $\frac{2}{n}$. It also proves an interesting result about the form of optimal strategies: every optimal strategy in the \emph{needle in a haystack} setting is such that every element $c\in\nset$ has an image which is present in all of the strategy's classes.
		\end{remark}   	
	\end{proof}
}


\junk{
	\section{Generalization}
	
	Of course, if Alice must whisper her message to Bob rather than use a transposition to transmit it, as is done in the two-party communication setting, then there is no reason why she must choose from exactly $n$ distinct messages. Perhaps we would like to add to or subtract from that number, making the game easier or harder, respectively, and call it the \emph{generalized two-party communication setting}.
	
	\begin{definition}
		The maximum attainable sum of intensities received from partitioning $\PER$ to $k$ parts is called the \textbf{$k$-field} of $\PER$, and is denoted $F(n,k)$. Fields are simply diagonal $k$-fields (fields of the form $F(n,n)$).
	\end{definition}
	
	We have $F(n,1)=n!$ (yielding a success probability of $\frac{1}{n}$, corresponding to no information from Alice) and $F(n,k)=nn! \hspace{1mm}\forall k\ge n!$ (yielding a success probability of $1$, corresponding to full information from Alice). Using exactly the same technique as in the proof of Theorem 3, we obtain that in general, for any natural number $r$, $F(n,k)=O(n!\frac{\log n}{\log\log n}+\frac{n!}{\frac{r\log n}{\log\log n}!}nk)$, with a success probability of $O(\frac{\log n}{n\log\log n}+n^{\frac{\log k}{\log n}+\frac{r-r\log r}{\log\log n}-r\frac{\log\log n-\log\log\log n}{\log\log n}})$. This last expression immediately yields the following theorem.
	
	\begin{theorem}
		In the generalized two-party communication setting with $n$ cards and lockers, if Alice has only $k=poly(n)$ messages to choose from, the maximum attainable probability is at most $\frac{(1+o(1))\log n}{n\log\log n})$. Furthermore, if we also have that $k\ge n$,\Mike{What should this bound on $k$ be?} \George{Well, Artur assumed $n$ distinct messages to derive the lower bound, I am not certain how many fewer would suffice.}\Artur{If we will want to use it then I'll give the formula for that.}then the maximum attainable probability is $\frac{(1+o(1))\log n}{n\log\log n}$.
	\end{theorem}
	
	Asymptotic results for several other interesting domains of $k$ could be found in this way. However, for super-polynomial domains, the upper bound derived in the above manner would most probably disagree with the lower bound of $\frac{(1+o(1))\log n}{n\log\log n}$ that we currently can provide, whereas, for sub-polynomial domains, we have no knowledge of a lower bound at all. Thus, the search for good asymptotic values for $F(n,k)$ continues. Determining some things about the rate of growth of $F(n,k)$ for fixed $n$ would be a good step in that direction. With this in mind, we close with a conjecture.\Artur{Why is this conjecture interesting?}\George{I added a part explaining that.}
	
	\begin{conjecture}
		For any fixed $n$, the function $f(k)=F(n,k)$ is concave.
	\end{conjecture}
	
	
	\section{Conclusions}
	\label{sec:conclusions}
	
	In this paper we presented a new locker room problem and provide a comprehensive analysis of its optimal strategy. The core of our analysis is a novel study of properties of random permutations with a given number of fixed points.
	
	There are several natural generalizations of the problem and related questions about properties of random permutations, and we will list here only a few.
	
	In the two-party communication setting from \Cref{subsec:com-compl}, when Alice sends her message $j$ to Bob, then there is no reason why she must choose number in $j \in \nset$; instead, she could transmit a number $j \in [m-1]$ for an arbitrary integer $m$. One can easily generalize the analysis from \Cref{thm:upper,thm:lower} in this setting for a large range of $m$. For example, even if $m = \text{poly}(n)$ then the maximum attainable success probability is still at most $\frac{(1+o(1)) \log n}{n \log\log n}$, and so the shift strategy is still asymptotically optimal. However, we do not have a complete tradeoff for the success probability of this process for superpolynomial value of $m$.
	
	We have also some optimal strategies for a variant of the locker problem in the two-party communication setting, but with a restriction that if $\Strat=\langle C_0,\dots,C_{n-1}\rangle$ is the agreed-upon strategy, then it must hold that
	\begin{align*}
	\exists_{c \in \nset} \ \forall_{k \in \nset} \ \exists_{j \in \nset} \ \forall_{\sigma \in C_j} \sigma(c) \ne k
	\enspace,
	\end{align*}
	that is, ``each possible image of the element $c$ is avoided by some class.'' In particular, we can show that the following simple \emph{freshman's strategy} that groups permutations according to the image of $0$ is optimal in this setting. (That is, $\sigma, \sigma'$ belong to the same class if and only if $\sigma(0) = \sigma'(0)$.) While for the original locker problem studied in this paper straightforward calculations yield a success probability of $\frac{2}{n}$ for the freshman's strategy, we can show that in the restricted setting described above the freshman's strategy is optimal.
}

\junk{
\section{Conclusions}
\label{sec:conclusions}

In this paper we presented a new search problem and provided a comprehensive analysis of its optimal strategy. The core of our analysis is a novel study of properties of random permutations with a given number of fixed points.\Artur{Do we want to write anything else here? As for now, it's unclear whether we need \Cref{sec:conclusions} at all.}
}



\newcommand{\Proc}{Proceedings of the~}
\newcommand{\DISC}{International Symposium on Distributed Computing (DISC)}
\newcommand{\FOCS}{IEEE Symposium on Foundations of Computer Science (FOCS)}
\newcommand{\ICALP}{Annual International Colloquium on Automata, Languages and Programming (ICALP)}
\newcommand{\IPCO}{International Integer Programming and Combinatorial Optimization Conference (IPCO)}
\newcommand{\ISAAC}{International Symposium on Algorithms and Computation (ISAAC)}
\newcommand{\JACM}{Journal of the ACM}
\newcommand{\NIPS}{Conference on Neural Information Processing Systems (NeurIPS)}
\newcommand{\OSDI}{Conference on Symposium on Opearting Systems Design \& Implementation (OSDI)}
\newcommand{\PODS}{ACM SIGMOD Symposium on Principles of Database Systems (PODS)}
\newcommand{\PODC}{ACM Symposium on Principles of Distributed Computing (PODC)}
\newcommand{\RANDOM}{International Workshop on Randomization and Approximation Techniques in Computer Science (RANDOM)}
\newcommand{\RSA}{Random Structures and Algorithms}
\newcommand{\SICOMP}{SIAM Journal on Computing}
\newcommand{\SIROCCO}{International Colloquium on Structural Information and Communication Complexity}
\newcommand{\SODA}{Annual ACM-SIAM Symposium on Discrete Algorithms (SODA)}
\newcommand{\SPAA}{Annual ACM Symposium on Parallel Algorithms and Architectures (SPAA)}
\newcommand{\STACS}{Annual Symposium on Theoretical Aspects of Computer Science (STACS)}
\newcommand{\STOC}{Annual ACM Symposium on Theory of Computing (STOC)}




\addcontentsline{toc}{section}{References}

\bibliographystyle{plainurl}

\bibliography{references}



\appendix
\begin{center}\huge\bf Appendix\end{center}



\junk{
\section{Formal framework and justification about worst-case vs. random \sought}
\label{sec:setup-random-sought}

\newcommand{\pp}{\ensuremath{\textcolor[rgb]{0.50,0.00,1.00}{\mathfrak{p}}}\xspace}
\renewcommand{\pp}{\ensuremath{\mathfrak{p}}\xspace}

In our analysis for the upper bounds in \ref{sec:prelim,sec:upper} (\ref{thm:upper}) and \ref{sec:locker-room-analysis} (\ref{Alice-Bob-vs-needle-haystack}), for simplicity, we have been making the assumption that \sought, the input to the \emph{needle in a haystack} search problem and to the \LRP, is random, that is, is chosen i.u.r. from \nset. (We do not make such assumption in the lower bound in \ref{sec:lower} (\ref{thm:lower}), where the analysis is done explicitly for arbitrary \sought.) In this section we will provide an elementary justification of why we were able to make such assumption about the random \sought, without any loss of generality.

We consider the problem with two inputs: a number $\sought \in \nset$ and a permutation $\sigma \in \PER$. We are assuming that $\sigma$ is a random permutation in \PER; no assumption is made about \sought.

For the \emph{needle in a haystack} search problem (a similar framework can be easily set up for the \LRP), a strategy (or an algorithm) is defined by a pair of two (possibly randomized) functions, $\hint = \hint(\sigma)$ and $\choos = \choos(\hint,\sought)$, with both $\hint, \choos \in \nset$.

For a fixed strategy, let $\pp(\sought)$ be the success probability for a given \sought and for a randomly chosen $\sigma \in \PER$. That is,
\begin{align*}
\pp(\sought) &= \Pr{\sigma(\choos) = \sought}
\enspace,
\end{align*}
where the probability is over $\sigma$ taken i.u.r. from \PER, and over the randomness in the choice of the strategy (since both $\hint = \hint(\sigma)$ and $\choos = \choos(\hint,\sought)$ may be randomized functions).

The goal is to design an algorithm (find a strategy) that will achieve some success probability for every $\sought \in \nset$.
That is, we want to have a strategy for which
\begin{align*}
\Pr{\mathcal{V}} &= \min_{\sought \in \nset}\{\pp(\sought)\}
\end{align*}
is maximized.

In our analysis for the upper bounds in \ref{sec:prelim,sec:upper} (\ref{thm:upper}) and \ref{sec:locker-room-analysis} (\ref{Alice-Bob-vs-needle-haystack}), the main claim (\ref{thm:upper}) is that if we choose \sought i.u.r. then $\pp(\sought) \le \frac{(1+o(1)) \log n}{n\log\log n}$, though in fact, one can read this claim as that $\sum_{\sought \in \nset} \frac{\pp(\sought)}{n} \le \frac{(1+o(1)) \log n}{n\log\log n}$. However, notice that this trivially implies that
\begin{align*}
\Pr{\mathcal{V}} &= \min_{\sought \in \nset}\{\pp(\sought)\} \le
\sum_{\sought \in \nset} \frac{\pp(\sought)}{n}
\enspace,
\end{align*}
and therefore \ref{thm:upper} yields $\Pr{\mathcal{V}} \le \frac{(1+o(1)) \log n}{n\log\log n}$, as required.

Observe that such arguments hold only for the upper bound. Indeed, since $\min_{\sought \in \nset}\{\pp(\sought)\}$ may be much smaller than $\sum_{\sought \in \nset} \frac{\pp(\sought)}{n}$, in order to give a lower bound for the success probability, \ref{thm:lower} proves that there is a strategy that ensures that $\pp(\sought) \ge \frac{(1+o(1)) \log n}{n\log\log n}$ for every $\sought \in \nset$; this clearly yields $\Pr{\mathcal{V}} \ge \frac{(1+o(1)) \log n}{n\log\log n}$, as required.

\junk{
	Now, we have two claims:
	
	Lower bound (\Cref{thm:lower}): There is a strategy that ensures that $\pp(\sought) \ge \frac{(1+o(1)) \log n}{n\log\log n}$ for every $\sought \in \nset$.
	
	Upper bound (\Cref{thm:upper}): For every strategy, if we choose \sought i.u.r. then $\pp(\sought) \le \frac{(1+o(1)) \log n}{n\log\log n}$, or more precisely, $\sum_{\sought \in \nset} \frac{\pp(\sought)}{n} \le \frac{(1+o(1)) \log n}{n\log\log n}$.
}
}


\section{Proofs of auxiliary claims}
\label{sec:appendix}


\subsection{Proof of \Cref{lemma:expected} (from \Cref{sec:lower})}
\label{proof:lemma:expected}

We will present an elementary (and following standard arguments) proof of \Cref{lemma:expected} showing that the expected number of $j \in \nset$ with $S_j \ge \frac{(1+o(1)) \cdot \log n}{\log\log n}$ is at least one.

\begin{proof}[Proof of \Cref{lemma:expected}]
Let us recall \Cref{def:derangements} for derangements and $r$-partial derangements. The probability that a random permutation in $\PER$ is a derangement is $D_n/n! = \lfloor{\frac{n!}{e}+\frac12}\rfloor/{n!} \sim \frac{1}{e}$.
	Let $u(n) = \lfloor \frac{n!}{e} + \frac12 \rfloor / \frac{n!}{e}$ and note that $D_n = u(n) \, n!/e$, that $u(n) = 1+o(1)$, and $u(n) > 0.9$ for all $n>1$. Since the permutation $\sigma \in \PER$ is chosen i.u.r., we have
	\begin{align*}
	\Pr{S_0 = k} &= \frac{D_{n,k}}{n!} = \frac{\binom{n}{k}D_{n-k}}{n!} = \frac{\binom{n}{k}\frac{(n-k)!}{e}u(n-k)}{n!}  = \frac{u(n-k)}{ek!}
	\enspace.
	\end{align*}
	
	The same bound can be obtained for $S_j$ for every $j \ge 0$. For any permutation $\sigma \in \PER$ and any integer $\ell \in \nset$, define permutation $\sigma_{\ell} \in \PER$ such that
	\begin{align*}
	\sigma_{\ell}(i) &=
	\sigma(i) + \ell \pmod n
	\enspace.
	\end{align*}
	For any permutation $\sigma \in \PER$ and any $\ell$, the operator $\sigma \mapsto \sigma_{\ell}$ is a bijection from $\PER$ to $\PER$, and a permutation $\sigma \in \PER$ with $\ell \in \nset$ has exactly $k$ fixed points if and only if permutation $\sigma_{\ell}$ has exactly $k$ points with $\sigma_{\ell}(i) = i + \ell \pmod n$. Hence for every $j, j' \in \nset$ and $k \in [n]$, we have $\Pr{S_j = k} = \Pr{S_{j'} = k}$.
	
	Therefore, for any integers $j \in \nset$ and $k \in [n-2]$,
	\begin{align}
	\label{ineq:prob-Sj}
	\Pr{S_j = k} &=
	\frac{u(n-k)}{ek!} > \frac{1}{2ek!}
	\enspace.
	\end{align}
	Let $k(n)$ be the largest $k$ such that $2ek! \le n$. Then $\Pr{S_j = k(n)} > 1/n$. Hence, if we let $Q_j$ be the indicator random variable that $S_j = k(n)$, then $\Pr{Q_j=1} > 1/n$, and hence $\Ex{\sum_{j=0}^{n-1} Q_j} = \sum_{j=0}^{n-1} \Ex{Q_j} = \sum_{j=0}^{n-1} \Pr{Q_j=1} > 1$. Therefore, in expectation, there is at least one value $j$ such that $S_j = k(n)$. It is easy to show that $k(n) = \frac{\log n}{\log\log n} (1+o(1))$.
\end{proof}


\subsection{Proof of \Cref{lemma:compatible-almost-all} (from \Cref{subsub:Disjoint-sets-compatible-for-shift})}
\label{proof:lemma:compatible-almost-all}

\begin{proof}[Proof of \Cref{lemma:compatible-almost-all}]
	$I$ and $J$ are compatible for shift $s$ if sets $I$, $J$, $I - s$, and $J + s$ are pairwise disjoint. 
	We will give a construction of sets $I$ and $J$, each of size $t$, such that $I$ and $J$ are compatible for shift $s$.
	
	
	We begin by selecting $t$ elements from $I$ one by one. We will ensure that sets $I$, $I-s$, and $I-2s$ are pairwise disjoint. The first element $i_1$ is arbitrary, and we can select it in $n$ ways. We choose the second element $i_2$ from $\nset \setminus \{i_1, i_1-s \pmod n, i_1-2s \pmod n\}$ in at least $n-3$ ways,
	the third element $i_3$ in at least $n-6$ ways, and so on; since the elements in $I$ can be ordered arbitrarily, the number of choices is at least $\frac{n (n-3) \dots (n-3(t-1))}{t!}$.
	
	Next, we choose $t$ elements from $J$. We will ensure that $J$ is pairwise disjoint from sets $I$ and $I-s$, and $J+s$ is pairwise disjoint from sets $I$ and $I-s$; notice that the latter means that $J$ is pairwise disjoint from sets $I-s$ and $I-2s$. The first element $j_1$ is selected in at least $(n-3t)$ ways, 
	since $j_1 \in \nset \setminus (I \cup I-s \cup I-2s)$ implies that $\{j_1\} \cap (I \cup I-s) = \emptyset$ and $\{j_1 + s \pmod n\} \cap (I \cup I-s) = \emptyset$. Next, we select $j_2 \in \nset \setminus (I \cup I-s \cup I-2s \cup \{j_1, j_1+s \pmod n, j_1-s \pmod n\})$ to ensure that the constructed $I$ and $J = \{j_1,j_2\}$ are compatible for shift $s$. Then we select $j_3 \in \nset \setminus (I \cup I-s \cup I-2s \cup \{j_1,j_2\} \cup \{j_1,j_2\}+s \cup \{j_1,j_2\}-s)$ in at least $(n-3(t+2))$ ways, and so on. Since the elements in $J$ can be ordered arbitrarily, the number of choices is $\frac{(n-3t) (n-3(t+1)) \dots (n-3(2t-1))}{t!}$.
	
	Therefore, we have presented a way of selecting at least
	\begin{align*}\frac{n (n-3) \dots (n-3(t-1))}{t!} \cdot \frac{(n-3t) (n-3(t+1)) \dots (n-3(2t-1))}{t!}\end{align*} distinct pairs of sets $I$ and $J$ of size $t$ that are compatible for shift $s$. This implies that if we choose two disjoint sets $I, J \subseteq \nset$ of size $t$ i.u.r., then the probability that $I$ and $J$ are compatible for shift $s$ is at least
	\begin{align*}
	\frac{\frac{n (n-3) \dots (n-3(t-1))}{t!} \cdot
		\frac{(n-3t) (n-3(t+1)) \dots (n-3(2t-1))}{t!}}
	{\binom{n}{t} \cdot \binom{n-t}{t}}&=
	\prod_{\ell=0}^{2t-1} \frac{(n-3\ell)}{(n-\ell)} =
	\prod_{\ell=0}^{2t-1} \left(1 - \frac{2\ell}{n-\ell}\right) \\&\ge
	\left(1 - \frac{4t}{n-2t}\right)^{2t} \ .
	\end{align*}
	Next, we use $\left(1-\frac{1}{a+1}\right)^a > e^{-1}$ to get $\left(1 - \frac{4t}{n-2t}\right)^{2t} > e^{\frac{-8t^2}{n-6t}}$ and then we use the assumption $t \le O(\log n)$ to get $e^{\frac{-8t^2}{n-6t}} \ge e^{-O(\log^2n) / n} \ge 1 - O\left(\frac{\log^2n}{n}\right)$.
\end{proof}


\subsection{Proof of \Cref{lemma:feasible-aux1} (from \Cref{subsubsec:Properties-feasible-sets})}
\label{proof:lemma:feasible-aux1}

\begin{proof}[Proof of \Cref{lemma:feasible-aux1}]
	Let $\zeta: K \rightarrow \{0,1\}$. We call a permutation $\sigma \in \PER$ consistent with $I$, $J$, $s$, $K$, and $\zeta$, if
	\begin{itemize}
		\item if $i \in I$ then $\sigma(i)$,
		\item if $j \in J$ then $\sigma(j) = j + s \pmod n$, and
		\item if $k \in K$ then $\sigma(k) = k + \zeta(k) \cdot s \pmod n$.
	\end{itemize}
	Let $\esetc_{I,J,0,s}^{\zeta}(K)$ be the set of all permutations consistent with $I$, $J$, $s$, $K$, and $\zeta$. We notice that $\eset_{I,J,0,s}(K)$ is the union over all $2^{|K|}$ functions $\zeta: K \rightarrow \{0,1\}$ of the sets of all permutations consistent with $I$, $J$, $s$, $K$, and $\zeta$, that is, $\eset_{I,J,0,s}(K) = \bigcup_{\zeta: K \rightarrow \{0,1\}} \esetc_{I,J,0,s}^{\zeta}(K)$.
	
	First, let us note that if $K$ is feasible for $I$, $J$, and $s$, then for any two distinct functions $\zeta, \zeta': K \rightarrow \{0,1\}$ the set of all permutations consistent with $I$, $J$, $s$, $K$, and $\zeta$ and the set of all permutations consistent with $I$, $J$, $s$, $K$, and $\zeta'$ are disjoint, that is, $\esetc_{I,J,0,s}^{\zeta}(K) \cap \esetc_{I,J,0,s}^{\zeta'}(K) = \emptyset$. Indeed, let us take two distinct $\zeta, \zeta': K \rightarrow \{0,1\}$ and let $\sigma$ be an arbitrary permutation in $\esetc_{I,J,0,s}^{\zeta}(K)$; we will show that $\sigma \not\in \esetc_{I,J,0,s}^{\zeta'}(K)$. Since $\zeta$ and $\zeta'$ are distinct, there is $\ell$ such that $\zeta(\ell) \ne \zeta'(\ell)$; without loss of generality let $\zeta(\ell) = 0$. But then, for any permutation $\sigma' \in \esetc_{I,J,0,s}^{\zeta'}(K)$ we have $\sigma'(\ell) = \ell + \zeta'(\ell) \cdot s \pmod n \ne \ell + \zeta(\ell) \cdot s \pmod n$, and thus $\sigma \not\in \esetc_{I,J,0,s}^{\zeta'}(K)$, and hence $\esetc_{I,J,0,s}^{\zeta}(K) \cap \esetc_{I,J,0,s}^{\zeta'}(K) = \emptyset$.
	
	Next, we argue that for any $\zeta: K \rightarrow \{0,1\}$, if $K$ is feasible for $I$, $J$, and $s$, then $|\esetc_{I,J,0,s}^{\zeta}(K)| = (n - |I \cup J \cup K|)!$. Indeed, for a given $\zeta: K \rightarrow \{0,1\}$, let $K + \zeta = \{k + \zeta(k) \cdot s \pmod n: k \in K\}$; let $\mathbb{S}_{I,J,K,s}(\zeta)$ be the set of all permutations $\pi^*: \nset \setminus (I \cup J \cup K) \rightarrow \nset \setminus (I \cup J+s \cup K + \zeta)$. Notice that since $K$ is feasible for $I$, $J$, and $s$, both
	\begin{enumerate}[(1)]
		\item $I$, $J$, and $K$ are pairwise disjoint, and
		\item $I$, $J+s$, and $K+\zeta$ are pairwise disjoint.
	\end{enumerate}
	Therefore $\mathbb{S}_{I,J,K,s}(\zeta)$ is non-empty, and hence $|\mathbb{S}_{I,J,K,s}(\zeta)| = (n - |I \cup J \cup K|)!$. Now, the claim that $|\esetc_{I,J,0,s}^{\zeta}(K)| = (n - |I \cup J \cup K|)!$ follows directly from the fact that any permutation consistent with $I$, $J$, $s$, $K$, and $\zeta$ corresponds in a unique way to a permutation in $\mathbb{S}_{I,J,K,s}(\zeta)$.\footnote{That is, for any $\sigma \in \PER$ consistent with $I$, $J$, $s$, $K$, and $\zeta$, and any $\sigma^* \in \mathbb{S}_{I,J,K,s}(\zeta)$, we define $\sigma' \in \PER$ such that
		\begin{align*}
		\sigma'(\ell) =
		\begin{cases}
		\ell & \text{ if } \ell \in I \cup J \cup K, \\
		\sigma^*(\ell) & \text{ if } \ell \in \nset \setminus (I \cup J \cup K).
		\end{cases}
		\end{align*}
	}
	
	We now summarize our discussion under the assumption that $K$ is feasible for $I$, $J$, and $s$. We have
	\begin{itemize}
		\item $\eset_{I,J,0,s}(K) = \bigcup_{\zeta: K \rightarrow \{0,1\}} \esetc_{I,J,0,s}^{\zeta}(K)$,
		\item for any $\zeta: K \rightarrow \{0,1\}$ it holds $|\esetc_{I,J,0,s}^{\zeta}(K)| = (n - |I \cup J \cup K|)!$, and
		\item for any two distinct functions $\zeta, \zeta': K \rightarrow \{0,1\}$ sets $\esetc_{I,J,0,s}^{\zeta}(K)$ and $\esetc_{I,J,0,s}^{\zeta'}(K)$ are disjoint.
	\end{itemize}
	This clearly implies that $|\eset_{I,J,0,s}(K)| = 2^{|K|} \cdot (n - |I \cup J \cup K|)!$.
\end{proof}


\subsection{Proof of \Cref{lemma:feasible-almost-all} (from \Cref{subsubsec:Properties-feasible-sets})}
\label{proof:lemma:feasible-almost-all}

\begin{proof}[Proof of \Cref{lemma:feasible-almost-all}]
	Following the approach from \Cref{lemma:compatible-almost-all}, for given disjoint sets $I$ and $J$ that are compatible for $s$, we will construct sets $K \subseteq \nset \setminus (I \cup J)$ that ensure that the constructed $K$ are feasible for $I$, $J$, and $s$.
	
	We select set $K \subseteq \nset \setminus (I \cup J)$ by choosing $k$ elements one by one. We will want to ensure that $K$ is pairwise disjoint with the sets $I$, $J$, $I-s$, $J+s$, and $K+s$. The first element $k_1$ is selected arbitrarily from $\nset \setminus (I \cup J \cup I-s \cup J+s)$ in at least $n - 4t$ ways. The second element cannot be in $I \cup J \cup I-s \cup J+s$ and also must be distinct from $k_1$ and $k_1 + s \pmod n$; hence, it can be chosen in at least $n - 4t - 2$ ways. In the same way, inductively, $k_{\ell}$ is selected from $\nset \setminus (I \cup J \cup I-s \cup J+s \cup \{k_r: 1 \le r < \ell\} \cup \{k_r + s: 1 \le r < \ell\})$ in at least $n - 4t - 2(\ell-1)$ ways. Since the elements in $K$ can be ordered arbitrarily, we constructed a set of at least $\frac{(n-4t) \dots (n-4t-2(k-1))}{k!}$ distinct sets $K \subseteq \nset \setminus (I \cup J)$ of size $k$ that are feasible for $I$, $J$, and $s$. Thus the probability that a set $K \subseteq \nset \setminus (I \cup J)$ of size $k$ chosen i.u.r. is feasible for $I$, $J$, and $s$ is at least
	\begin{align*}
	\frac{\frac{(n-4t) \dots (n-4t-2(k-1))}{k!}}{\binom{n-2t}{k}} &=
	\prod_{\ell=0}^{k-1} \frac{n-4t-2\ell}{n-2t-\ell} =
	\prod_{\ell=0}^{k-1} \left(1 - \frac{2t+\ell}{n-2t-\ell}\right) \ge
	\prod_{\ell=0}^{k-1} \left(1 - \frac{2t+k}{n-2t-k}\right) \\&=
	\left(1 - \frac{2t+k}{n-2t-k}\right)^k
	\enspace.
	\end{align*}
	Next, assuming that $t, k \le O(\log n)$, we have $\left(1 - \frac{2t+k}{n-2t-k}\right)^k \ge e^{\frac{-(2t+k)k}{n-4t-2k}}
	\ge 1 - O\left(\frac{\log^2n}{n}\right)$.
\end{proof}

\subsection{Proof of \Cref{claim:bound-for-feasible-K} (from \Cref{subsubsec:Approximating-set-for-compatible-sets})}
\label{proof:claim:bound-for-feasible-K}

\begin{proof}[Proof of \Cref{claim:bound-for-feasible-K}]
	Let $\epsilon$ be such that the $1 - O\left(\frac{\log^2n}{n}\right)$ probability in \Cref{lemma:feasible-almost-all} is at least $1 - \epsilon$. For simplicity of notation, let
	\begin{align*}
	A_k = \{K \subseteq \nset \setminus (I \cup J): |K|= k \text{ and $K$ is feasible for $I$, $J$, and $s$}\}
	\enspace.
	\end{align*}
	Next, notice that by combining \Cref{lemma:feasible-almost-all} with the trivial upper bound for $|A_k|$, we have
	\begin{align}
	\label{ineq:feasible-almost-all}
	(1 - \epsilon) \cdot \binom{n-2t}{k} &\le |A_k| \le \binom{n-2t}{k}
	\enspace.
	\end{align}
	Then, we have,
	\begin{align}
	\sum_{k=1}^{2r} (-1)^{k+1}\!\! & \sum_{K \in A_k} |\eset_{I,J,0,s}(K)| \\
	&=^{\text{(by \Cref{lemma:feasible-aux1}) \ }}
	\sum_{k=1}^{2r} (-1)^{k+1} \sum_{K \in A_k} 2^k \cdot (n - 2t - k)!
	\nonumber\\
	&\ge^{\text{(by (\ref{ineq:feasible-almost-all}))}}
	\sum_{k=1}^{2r} 2^k \cdot (n - 2t - k)! \cdot \binom{n-2t}{k} \cdot
	\begin{cases}
	(1 - \epsilon)&\text{if } k \text{ odd}\\
	-1&\text{if } k \text{ even}\\
	\end{cases}
	\nonumber\\
	&=
	(n - 2t)! \left(-\sum_{k=1}^{2r} \frac{(-2)^k}{k!} - \epsilon\mkern-18mu \sum_{k=1,\ k \text{ odd}}^{2r} \frac{2^k}{k!} \right)
	\nonumber\\
	&\ge
	(n - 2t)! \left(1-\sum_{k=0}^\infty \frac{(-2)^k}{k!} -\frac{2^{2r}}{(2r)!}- \epsilon \sum_{k=0}^\infty \frac{2^k}{k!} \right)
	\label{ineq:2r-up}\\
	&= (n - 2t)!\ (1 - e^{-2} -\frac{2^{2r}}{(2r)!} - \epsilon \cdot e^2)
	\label{eq:e-sum}
	\enspace.
	\end{align}
Inequality (\ref{ineq:2r-up}) holds because $\frac{2^{2r}}{(2r)!}\ge \frac{2^{2r+1}}{(2r+1)!}$ for all $r>0$.
Equality~\ref{eq:e-sum}) holds since $\sum_{k=0}^{\infty} \frac{(-2)^k}{k!} = e^{-2}$ and $\sum_{k=0}^{\infty} \frac{2^k}{k!} = e^2$.
Inequality~(\ref{ineq:bound-for-feasible-K}) follows at once since $2r\geq \log_2 n$ and $(\log n)!=n^{\Omega(\log\log n)}$.
\end{proof}


\subsection{Proof of \Cref{claim:bound-for-infeasible-K} (from \Cref{subsubsec:Approximating-set-for-compatible-sets})}
\label{proof:claim:bound-for-infeasible-K}

\begin{proof}[Proof of \Cref{claim:bound-for-infeasible-K}]
	For simplicity of notation, let
	\begin{align*}
	N\!A_k &=
	\{K \subseteq \nset \setminus (I \cup J): |K|= k \text{ and $K$ is \emph{not} feasible for $I$, $J$, and $s$}\}
	\enspace.
	\end{align*}
	In our analysis we use two basic facts for sets $K \in N\!A_k$: that $|\eset_{I,J,0,s}(K)| \le 2^k (n-2t-k)!$ and that the set of such $K \subseteq \nset \setminus (I \cup I)$ is by \Cref{lemma:feasible-almost-all}, of size at most $O\left(\frac{\log^2n}{n}\right) \cdot \binom{n-2t}{k}$:
	\begin{align*}
	\sum_{k=1}^{2r} (-1)^{k+1}
	\sum_{K \in N\!A_k} |\eset_{I,J,0,s}(K)|
	&\ge
	- \sum_{k=1}^{2r}
	\sum_{K \in N\!A_k} |\eset_{I,J,0,s}(K)|
	\ge
	- \sum_{k=1}^{n-2t}
	\sum_{K \in N\!A_k} |\eset_{I,J,0,s}(K)|
	\\&\ge
	- \sum_{k=1}^{n-2t}
	\sum_{K \in N\!A_k} 2^k (n-2t-k)!
	\\&\ge
	- \sum_{k=1}^{n-2t}
	O\left(\frac{\log^2n}{n}\right) \binom{n-2t}{k}  2^k \cdot (n-2t-k)!
	\\&=
	- O\left(\frac{\log^2n}{n}\right) (n-2t)!  \sum_{k=1}^{n-2t} \frac{2^k}{k!}
	\\& \ge
	- O\left(\frac{\log^2n}{n}\right)  (n-2t)! \ \sum_{k=1}^{\infty} \frac{2^k}{k!}
	\\&=
	- O\left(\frac{\log^2n}{n}\right)  (n-2t)!  e^2
	=
	- O\left(\frac{\log^2n}{n}\right)  (n-2t)! \ .
	\end{align*}
\end{proof}


\end{document}